\documentclass[english,12pt]{article}
\usepackage[T1]{fontenc}
\usepackage[latin9]{inputenc}
\usepackage{geometry}
\usepackage{amsfonts,latexsym,amsthm,amsxtra,amssymb,bbm}
\usepackage{babel}

\usepackage{verbatim}
\usepackage{float}
\usepackage{bm}
\usepackage{amsmath}
\usepackage{amssymb}
\usepackage{graphicx}
\usepackage{float}
\usepackage{subfigure}
\usepackage{epstopdf}

\usepackage{breakurl}

\RequirePackage[colorlinks,citecolor=blue,urlcolor=blue]{hyperref}
\RequirePackage{natbib}
\makeatletter

\floatstyle{ruled}
\newfloat{algorithm}{tbp}{loa}
\providecommand{\algorithmname}{Algorithm}
\floatname{algorithm}{\protect\algorithmname}

\usepackage{amsthm}\usepackage{dsfont}\usepackage{array}\usepackage{mathrsfs}\usepackage{comment}\onecolumn

\usepackage{color}\usepackage{babel}

\usepackage{enumitem}
\setlist[itemize]{leftmargin=1em}
\setlist[enumerate]{leftmargin=1em}
\allowdisplaybreaks

\usepackage{babel}
\usepackage[algo2e]{algorithm2e}
\usepackage{algorithmic}
\usepackage{arydshln}
\usepackage{enumitem}

\newcommand{\RR}{\mathbb{R}}

\definecolor{yxc}{RGB}{255,0,0}
\definecolor{cm}{RGB}{0,0,200}
\definecolor{kzw}{RGB}{0,150,0}
\def\ave{\textsf{ave}}

\makeatother

\newtheorem{condition}{Condition}
\theoremstyle{assumption}

\usepackage{setspace}

\begin{document}
\theoremstyle{plain} \newtheorem{lemma}{\textbf{Lemma}} \newtheorem{prop}{\textbf{Proposition}}\newtheorem{theorem}{\textbf{Theorem}}\setcounter{theorem}{0}
\newtheorem{corollary}{\textbf{Corollary}} \newtheorem{assumption}{\textbf{Assumption}}
\newtheorem{example}{\textbf{Example}}
 \newtheorem{definition}{\textbf{Definition}}
\newtheorem{fact}{\textbf{Fact}} \theoremstyle{definition}
 \newtheorem{con}{\bf Condiction}
  \newtheorem{defi}{\bf Definition}

\theoremstyle{remark}\newtheorem{remark}{\textbf{Remark}}

\title{{Maximum likelihood estimation in the sparse Rasch model\thanks{The authors are listed in the alphabetical order.}}}

\author
{
Pai Peng\thanks{Department of Artificial Intelligence, Jianghan University, Wuhan, 430056, China.
\texttt{Emails:}pengpai@jhun.edu.cn.} \quad
Lianqiang Qu\thanks{Department of Statistics, Central China Normal University, Wuhan, 430079, China.
\texttt{Emails:}qulianq@ccnu.edu.cn.}
\quad
Qiuping Wang\thanks{Department of Statistics, Zhaoqing University, Zhaoqing, 526000, China.
\texttt{Emails:}qp.wang@mails.ccnu.edu.cn.}
\quad
	Shufang Wang\thanks{Department of Statistics, Central China Normal University, Wuhan, 430079, China.
\texttt{Emails:}Wangshufang2050@163.com}
\quad
Ting Yan\thanks{Department of Statistics, Central China Normal University, Wuhan, 430079, China.
\texttt{Emails:}tingyanty@mails.ccnu.edu.cn.} \quad
}


\maketitle
\begin{abstract}
The Rasch model has been widely used to analyse item response data in psychometrics and educational assessments.
When the number of individuals and items are large, it may be impractical to provide all possible responses.
It is desirable to study sparse item response experiments.
Here, we propose to use the Erd\H{o}s\textendash R\'enyi random sampling design,
where an individual responds to an item with low probability $p$.
We prove the uniform consistency of the maximum likelihood estimator 
when both the number of individuals, $r$, and the number of items, $t$, approach infinity.
Sampling probability $p$ can be as small as  $\max\{\log r/r, \log t/t\}$ up to a constant factor, which is
a fundamental requirement to guarantee the connection of the sampling graph by the theory of the Erd\H{o}s\textendash R\'enyi graph.
The key technique behind this significant advancement is a powerful
leave-one-out method for the Rasch model.
We further establish the asymptotical normality of the MLE by using a simple matrix to approximate the inverse of the Fisher information matrix.
The theoretical results are corroborated by simulation studies and an analysis of a large item-response dataset.

\end{abstract}

\section{Introduction}

Item response experiments are used in many fields, including psychometrics, educational assessments, and
health-related quality of life,  among others \cite[e.g.][]{birdsall2011implementing,chen2023SS,rutkowski2013handbook},
where a number of individuals are asked to respond to a set of items (e.g. test questions).
The Rasch model \citep{georg1960probabilistic} was one of the earliest item response theory models used to characterise the difficulties of items (i.e. latent traits) and the abilities of individuals.
It assumes that the probability of individual $i$ giving a correct response to item $j$ is:
\begin{equation}\label{Rasch-pij}
\mathbb{P}(~\mbox{individual $i$~correctly answers item $j$}~)=\frac{e^{\alpha_i-\beta_j}}{1+e^{\alpha_i-\beta_j}},
\end{equation}
where $\alpha_i$ and $\beta_j$ denote the ability parameter for individual $i$ and difficulty parameter for item $j$, respectively.
A larger $\beta_j$ indicates that it is more difficult to provide a correct response to item $j$.
Because of its apparent simplicity and ease of interpretation, the Rasch model has been widely employed to analyse item response data \cite[e.g.][]{fischer1978probabilistic,wright2004overview,bond2013applying}.

Since each individual or each item is assigned one intrinsic parameter,
the number of parameters increases with the number of individuals or items.
Therefore, statistical inference is non-standard in the Rasch model.
This has attracted significant interest from statisticians and psychologists for the investigation of its theoretical properties
\cite[e.g.][]{de1986maximum,fischer1995derivations,gurer2023penalization}.
\citet[][pages 261-263]{fischer1974einfuhrung} and \citet{haberman1977maximum} derived the necessary 
and sufficient condition for the existence and uniqueness of the
maximum likelihood estimator (MLE).
However, when the number of individuals or items is fixed, MLE suffers from inconsistency problems \citep{andersen1973conditional,ghosh1995inconsistent}.
Conditional or marginal maximum likelihood procedures have been developed to address this problem
\citep{andersen1972numerical,follmann1988consistent,agresti1993computing,robitzsch2021comprehensive}.
In a pioneering study, \citet{haberman1977maximum} established the consistency and asymptotical normality of the MLE under the condition that
each individual gives responses to all items,
when both the number of individuals and items simultaneously go to infinity.

The Rasch model is closely related to the $\beta$-model of undirected graphs
\citep{Chatterjee2011random,Yan2013clt}, the $p_0$ model for directed graphs \citep{yan2016asymptotics}, the bipartite
$\beta$-model of affiliation networks \citep{fan2023asymptotic, wang2022two},
 and the Bradley--Terry model for paired comparisons \citep{bradley-terry1952}, in which all probability distributions for binary responses have a logistic regression representation.
In particular, sparse paired comparisons designs under the Erd\H{o}s\textendash R\'enyi random sampling schedule are studied in the Bradley--Terry model
\citep{chen2019spectral,han2020asymptotic,chen2022partial}, where the uniform consistencies of the MLE or the regularised MLE are established.

When the number of individuals, $r$, and the number of items, $t$, are large,  it may be impractical to provide all possible responses.
As an illustrating example, the Riiid AIEd train data set \citep{riiid-test-answer-prediction} used in the AAAI-2021 workshop on AI Education contains  $393,656$ individuals and $10,000$ items,
where the response rate is only $2.6\%$.
This issue is referred to as sparsity, which poses significant challenges in the theoretical analysis of the MLE under the Rasch model.
To address this issue,
we consider an Erd\H{o}s\textendash R\'enyi response design in the Rasch model,
where an individual gives response to an item with a low probability $p$.
This is called the {\it sparse Rasch model}. With this design, the total number of experiments is reduced to $rtp$.

The contributions of this study to the literature are as follows.
First, we prove the uniform consistency of the MLE by developing a leave-one-out method to analyse the
properties of the MLE and regularised MLE in the Rasch model,
where the smallest sampling probability $p$ can be as small as $\log r/r$ up to a constant factor.
By the theory of the Erd\H{o}s\textendash R\'enyi graph, this is a fundamental requirement for guaranteeing  the connection of the sampling graph  with high probability,
which is a necessary condition for the MLE existence.
Second, we show that the MLE converges to a normal distribution by using a simple matrix to approximate
the inverse of the Fisher information matrix.
Finally, we perform numerical simulations to evaluate the performance of MLE in finite samples, which agree with our theories.
We also apply our theoretical results to analyse large item response data.

It is noted that \cite{JMLR-Chen-2023} derived consistency and asymptotical normality of the MLE under a general missing-entry setting that does
not require a random sampling scheme\footnote{We noticed this paper after we finish our work, whose Chinese version is Shufang Wang's master thesis beginning in 2021.}.
When applying their results to the Erd\H{o}s\textendash R\'enyi sampling design, it is required that $p\gg (\log r/r)^{1/2}$ as mentioned in their paper.
Here, our results are much sharper, allowing that $p$ is close to the Erd\H{o}s\textendash R\'enyi lower bound.
On the other hand, their asymptotical variances are different from ours since the conditions of model identification are different.
We shall elaborate the differences in more details after presenting our main results (see Remarks \ref{remark-a} and \ref{remark-b}).
Additionally, our proof strategies are different from theirs: (1) we use the leave-one-out method to prove consistency while \cite{JMLR-Chen-2023} employs the fixed point theorems of
\citeauthor{akilov1964functional} (\citeyear{akilov1964functional}, pages 695-711); (2) we obtain asymptotic normality by approximating
the inverse of the Fisher information matrix while \cite{JMLR-Chen-2023} uses a complex three-way decomposition technique of the coefficients
lying in a constrained solution space.
We also note another related study.
\cite{yang2024randompairing} proposed random pairing maximum
likelihood estimator by transforming individual-item responses to item-item comparisons by randomly pairing responses of the same user to different items, and proved its consistency.
This work is different from ours, where we directly work on the maximum likelihood estimation.
In addition, the asymptotic distribution is not investigated in \cite{yang2024randompairing}.

The remainder of this paper is organised as follows.
Section \ref{sec-main} presents the main results, including
the MLE in Section \ref{subsec-mle}, its consistency in Section \ref{subsec-consis},
and its asymptotic normality in Section \ref{subsec-asym-norm}.
Numerical experiments are conducted to validate the theoretical findings in Section \ref{sec-simulation}.
Further discussions are presented in Section \ref{sec-discussion}.
The proofs for our main results are given in Section \ref{sec-appendix}, while the proofs of the supported lemmas are presented in the Supplementary Material.

{\emph{Notations.}}
For an integer $n$, we use $[n]$ to denote set $\{1,2,\ldots, n\}$.
For two positive numbers $a_n$ and $b_n$,  $a_n=O\left(b_n\right)$ or $a_n\lesssim b_n$ indicate that $a_n\le C b_n$ for a fixed constant $C>0$,
 $a_n=\Omega\left(b_n\right)$ or $a_n\gtrsim b_n$ denotes $a_n\ge Cb_n$;  $a_n=\Theta\left(b_n\right)$ or $a_n\asymp b_n$ denotes $a_n\lesssim b_n$; and $a_n\gtrsim b_n$;  $a_n=o(b_n)$ denotes  $\lim_{n\rightarrow\infty}\frac{a_n}{b_n}=0$.
For vector $x \in {\mathbb{R}^n}$, denote $L_1$-norm, $L_2$-norm, and $L_\infty$-norm by ${\left\| x \right\|_1} = \sum\nolimits_{i = 1}^n {\left| {{x_i}} \right|}$, $\left\| x \right\|_2^2 = \sum\nolimits_{i = 1}^n {x_i^2} $, and
${\left\| x \right\|_\infty } = {\max _{1 \le i \le n}}\left| {{x_i}} \right|$, respectively.
For a matrix $A=(a_{ij})_{i,j\in[n]}\in\mathbb{R}^{n\times n}$, we denote
${\left\| A \right\|_{\max}}:=\max\nolimits_{i,j\in[n]} |a_{ij}|$ by the maximum absolute entry-wise norm of matrix $A$.
Let $\mathds{1}_n$ be an $n$-dimensional column vector of all vectors, and
$\left\{\bm{e}_{i}\right\}_{1\leq i\leq n}$ denote the standard basis vector of $\RR^n$.
For any $x\in\mathbb{R}^n$, we write $\ave(x)=n^{-1}\mathds{1}_{n}^{\top}x$.
We use $C_i$ and $c_i$, $i=0,1,\ldots$, to denote universal constants that may change from place to place.

\section{Main results}
\label{sec-main}

Assume that there are $r$ individuals and $t$ items engaged in response experiments.
Let $[r]=\{1,\ldots, r\}$ and $[t]=\{1,\ldots, t\}$ denote sets of individuals and items, respectively.
We consider an Erd\H{o}s\textendash R\'enyi random response design, where we treat whether an individual responds to an item as a Bernoulli random variable with a successful probability $p$. Here, $p$ may depend on $r$ or $t$; the subscript is suppressed.
We recast all individuals and items into a bipartite graph  with a node set $[r+t]$, where the first $r$ nodes, that is, $1,\ldots, r$, represent individuals, and the remaining $t$ nodes, $r+1, \ldots, r+t$, represent items.
Let $X_{i,j+r}\in\{0,1\}$ be an indicator variable that denotes  whether individual $i$ responds to item $j$. If individual $i$ answers item $j$, then $X_{i,j+r}=X_{j+r,i}=1$;
otherwise, $X_{i,j+r}=X_{j+r,i}=0$. Note that $X_{i,j+r} \sim \mathrm{Bernoulli}(p)$.
We collect these in a matrix $X=\{X_{i,j}\}_{i,j\in [r+t] }$, which can be viewed as the adjacency matrix of an Erd\H{o}s\textendash R\'enyi bipartite random graph $G_{r,t}$. Let $\mathcal{E}$ be the collection of all edges in graph $G_{r,t}$.
Remarkably, $X_{i,j}=0$ for all $1\le i,j\le r$ and $r+1\le i, j \le r+t$.

Let $A=(a_{i,j})_{ (r+t) \times (r+t)}$ denote the outcome matrix recording information on whether individuals give correct responses to items, conditional on response graph $G_{r,t}$.
That is, if individual $i$ correctly answers item $j$, then $a_{i,j+r}=1$; otherwise, $a_{i,j+r}=0$.
In a random-sampling scenario, the Rasch model assumes that
all responses are independent
Bernoulli random variables with:
\begin{equation}
\label{eq:rasch model}
\mathbb{P} \left( {{a_{i,j}} = 1|{X_{i,j}} = 1} \right) = \frac{{{e^{{\alpha _i} - {\beta _{j }}}}}}{{1 + {e^{{\alpha _i} - {\beta _{j }}}}}},
\end{equation}
given response graph $G_{r,t}$.
As previously mentioned, $\alpha_i$ measures the ability of individual $i$ and $\beta _{j}$ the difficulty of item $j$.
The larger $\alpha_i$ is, the easier it is for individual $i$ to answer items, and vice versa.
Similar to the definition of $X_{i,j}$, note that $a_{i,j}=0$ for all $1\le i,j\le r$ and all $r+1\le i, j \le r+t$.

For convenience, we write:
\[
\theta=\left(\theta_{1}, \cdots, \theta_{r+t}\right)^{\top}=(\alpha_1, \ldots, \alpha_r, \beta_{1}, \ldots, \beta_{t})^\top.
\]
The Rasch model implies that the probability of a correct response depends only on the differences between an individual's ability parameters and an item's difficulty parameters. Specifically,
adding the same constant to $\alpha_i$ and $\beta_j$ results in the same probability, as in \eqref{eq:rasch model}. The following model identification conditions are required:
\begin{equation}
\label{eq:model-iden-a}
\theta_1 = 0,
\end{equation}
or
\begin{equation}
\label{eq:model-iden-b}
\sum_{i=1}^{r+t} \theta_i = 0.
\end{equation}
These two conditions are equivalent in that they can be re-parameterized into each other without changing the probability in \eqref{eq:rasch model}.

\subsection{Maximum likelihood estimation}
\label{subsec-mle}

We write:
\[
\mu(x) = \frac{ e^x }{ 1 + e^{x} }.
\]
Under sparse Rasch model \eqref{eq:rasch model}, the negative log-likelihood function conditional on $\mathcal{G}_{r,t}$ is given by:
\begin{equation}
\label{negative-likelihood-fu}
\begin{array}{rcl}
\ell (\theta ) &  = &  - \sum\limits_{~~\scriptstyle(i,j+r) \in {\cal E},\hfill\atop
\scriptstyle i \in [r],j \in [t]\hfill} {\left\{ {\left. {{a_{i,j+r}}\log \mu ({\alpha _i} - {\beta_{j}}) + (1 - {a_{i,j+r}})\log (1 - \mu ({\alpha_i} - {\beta _{j}})} \right\}} \right.}
\\
&  = &  - \sum\limits_{~~\scriptstyle(i,j+r) \in {\cal E},\hfill\atop
\scriptstyle i \in [r],j \in [t]\hfill} {\left\{ {\left. {{a_{i,j+r}}\log \mu ({\theta _i} - {\theta _{j+r}}) + (1 - {a_{i,j+r}})\log (1 - \mu ({\theta _i} - {\theta _{j + r}})} \right\}} \right.}.
\end{array}
\end{equation}
We define the MLE that minimises the negative log-likelihood function $\ell (\theta)$ as:
\begin{equation}
\label{def-mle}
\hat{\theta }=\underset{\theta \in {{\mathbb{R}}^{r+t}}}{\mathop{\arg \min }}\,{\ell  }(\theta ).
\end{equation}
Therefore, if $\hat{\theta}$ exists, then it must satisfy maximum likelihood equations:
\renewcommand{\arraystretch}{1.5}
\begin{equation}
\label{eq-mle-equa}
\begin{array}{rcl}
\sum_{j \in [t], (i,j+r)\in {\cal E} } a_{ij}   & = & \sum_{j \in [t], (i, j+r)\in {\cal E} }
a_{ij} \mu( \hat{\theta}_i - \hat{\theta}_{j+r}), ~~i \in [r]
\\
\sum_{i \in [r], (i,j)\in {\cal E} }  a_{ij}  & = &  \sum_{i \in [r], (i,j)\in {\cal E} } a_{ij} \mu( \hat{\theta}_i - \hat{\theta}_j), ~~j \in \{r+1,\ldots, r+t\}.
\end{array}
\end{equation}

\subsection{Consistency of the MLE}
\label{subsec-consis}

We explain below the basic conditions for this study.

\begin{condition}
\label{condi-t-r}
There are two fixed positive numbers $c_1$ and $c_2$ such that $c_2 \le  t/r  \le c_1$.
Furthermore, without loss of generality, we assume $t\ge r$ throughout this study.
\end{condition}

\begin{condition}
\label{condi-pr}
The sampling probability satisfies $p \geq \frac{c_{0}\log r}{r}$,
where $c_0$ denotes a sufficiently large constant.
\end{condition}

Condition \ref{condi-t-r} is mild in several item response situations. For instance,
there are $2,587$ people and $2,125$ items in the DuoLingo data set in \cite{wu2020variational},
where $r/t=1.2$.
Condition $t\ge r$ is used only to simplify the proofs, and all results still hold when $t<r$ under Condition \ref{condi-t-r}.

Condition \ref{condi-pr} is a fundamental requirement to guarantee the good properties of the MLE.
According to the theory of the Erd\H{o}s--R\'enyi graph \citep{erdHos1960evolution},  random bipartite graph $G_{r,t}$ is disconnected with a high probability when $p < [(1-\epsilon) {\log r}]/{r}$ for any $\epsilon>0$.
In a disconnected $G_{r,t}$, at least one of the following two cases exists:
(1) induced subgraphs $G_1$ containing nodes $N_{r1}\subset [r]$ and $N_{t1}\subset [t]$
and  $G_2$ containing nodes $N_{r2}\subset [r]$ and $N_{t2}\subset [t]$
 have no intersecting edges  for the two disjoint sets of individuals $N_{r1}$ and $N_{r2}$ and two disjoint sets of items $N_{t1}$ and $N_{t2}$; and
 (2) at least one isolated node with no edges connects to any other node in  $G_{r,t}$.
In the first case, multiplying the parameters associated with $G_1$ by a positive constant and dividing the parameters associated with $G_1$ by the same constant reduces the negative log-likelihood function such that the MLE does not exist. For the latter case,
the nonexistence of the MLE is clear.

We define $\theta_{\max}=\max_{i\in [r+t] } \theta_i$ and $\theta_{\min}=\min_{j\in [r+t] } \theta_j$.
Additionally, we define $\kappa = \theta_{\max} - \theta_{\min}$. We use superscript ``*" to denote the true parameter, that is,
$\theta^* \in \mathbb{R}^{r+t}$ denotes the true parameter vector of $\theta$.
The consistency of the MLE can be described as follows.

\begin{theorem}[{\bf Existence and entrywise error of the MLE}]
\label{th:MLE-main}
Under Conditions \ref{condi-t-r} and \ref{condi-pr}, if $\kappa = O(1)$, then
as $r\to\infty$, $\hat\theta$ defined in \eqref{def-mle} exists and satisfies:
\begin{eqnarray}
\label{eq:theorem-con-a}
\|  \hat{\theta}   - \theta^* - \mathsf{ave}(\hat{\theta}   - \theta^*) \|_2^2  &  \lesssim &  \frac{\log r}{p},
\\
\label{eq:theorem-con-b}
\max_{i\in[r+t]}|\hat\theta_i-\theta_i^* - \mathsf{ave}(\hat{\theta}   - \theta^*)|^2  & \lesssim &  \frac{\log r}{r p }
\end{eqnarray}
with a probability of at least $1-O(r^{-7})$, where $\mathsf{ave}(x)$ denotes the average value of vector $x$.
Furthermore, if $\hat\theta$ exists, it is unique.
\end{theorem}

\begin{remark}
\label{remark-a}
\cite{JMLR-Chen-2023} studied the Rasch model under general missing data mechanisms.
Specifically, they assume that $X_{i,j}$ denotes the missing status of entry $a_{i,j}$, where $X_{i,j} = 1$ indicates that $a_{i,j}$ is observed and
$X_{i,j} = 0$ otherwise. Therefore, their missing data mechanisms allow sparse item response experiments.
We compare their consistency results with ours. 
In Theorem 4 of \cite{JMLR-Chen-2023}, they give sufficient conditions for guaranteeing the consistency of the MLE
under the same Erd\H{o}s--R\'enyi sampling scheme, where $X_{i,j}$ are i.i.d. Bernoulli random variables with $\mathbb{P}(X_{i,j} =
1) = p$. To compare their conditions with ours, we reproduce their conditions: (1) $r p \ge  tp \ge (\log(r))^4$;
(2) for an integer $n \ge 1$, $p^n t^{(n-1)/2} r^{(n-1)/2} - \log (rt) \to \infty$
if $n$ is odd, and $p^n t^{n/2} r^{(n/2)-1} - 2\log r \to\infty$ if $n$ is even.
As mentioned in their papers, consider the setting $r = t$ and let $n = 2$, the above conditions requires $p^2r - 2 \log(r) \to \infty$, which implies
$p \gg \sqrt{\log r/ r}$. In addition, their condition can not allow $n=1$ since $p\le 1$. On the other hand, for larger $n$, their conditions imply that the restriction on $p$ will become stronger.
In sharp contrast to their conditions, we allow that $p$ is close to the Erd\H{o}s--R\'enyi lower bound ($p\ge c_0\log r/r$ for some constant $c_0$), where this is the smallest probability to guarantee that
the sampling graph is connected with high probability.
\end{remark}

\begin{remark}
When $p=1$, the consistency rates of the MLEs are the same order of $(\log r/r)^{1/2}$, which
matches the optimal convergence rates as discussed in \cite{JMLR-Chen-2023}.
Additionally, this also matches
the oracle inequality $\|\hat{\beta} - \beta\|_\infty = O_p(  (\log p/N)^{1/2} )$ for the Lasso estimator
in the linear model with $p$-dimensional true parameters $\beta$ and the sample size $N$ (e.g., \cite{Lounici2008}).
The assumption $\kappa=O(1)$ is made in \cite{JMLR-Chen-2023}, and \cite{chen2022partial} for studying the Bradley--Terry model for paired comparisons.
\end{remark}

\subsection{Asymptotic normality of the MLE}
\label{subsec-asym-norm}

Here, we use $\theta_1=0$ in \eqref{eq:model-iden-a} for model identification.
This can be achieved by subtracting $\theta_1$ from each element in $\theta$.
Let $V=(v_{ij})_{i,j\in[r+t]\backslash\{1\}}$ be a Hessian matrix of $\ell(\theta)$ with respect to $\theta_1=0$,
where
\begin{equation*}
\begin{array}{c}
v_{i,i} = \sum\limits_{j \in [t] } X_{i,j+r} \frac{ e^{\theta_i - \theta_{j+r}} }{ (1 + e^{\theta_i - \theta_{j+r}} )^2 },
~ i\in [r],
\quad
v_{j+r,j+r} = \sum\limits_{i \in [r] } X_{i,j+r} \frac{ e^{\theta_i - \theta_{j+r}} }{ (1 + e^{\theta_i - \theta_{j+r}} )^2 },
~ j\in [t],
\\
v_{i,j+r} = -  {X_{i,j+r}}\frac{ e^{\theta_i - \theta_{j+r}} }{ (1 + e^{\theta_i - \theta_{j+r}} )^2 },~ i\in [r], j\in [t],
\\
v_{i,j}=0, ~ 1\le i\neq j \le r; r+1 \le i\neq j \le r+t.
\end{array}
\end{equation*}
Remarkably, matrix $V$ is also a Fisher information matrix for parameter vector $(\theta_2, \ldots, \theta_{r+t})$.
The asymptotic variance of $\hat{\theta}$ depends crucially on the inverse of $V$.
However, $V^{-1}$ does not have an explicit solution in general.
We use a simple matrix $S$ to approximate $V^{-1}$, where $S=(s_{ij})_{i,j\in [r+t]\backslash \{1\}}$ and
 $s_{ij}=\delta_{ij}/v_{ii} + 1/v_{11}$. Here, $\delta_{ij}$ denotes the Kronecker delta function, where
 $\delta_{ij}=1$ if $i=j$ and otherwise $\delta_{ij}=0$.
 The approximation error of $V^{-1} -S$ in terms of the absolute entrywise norm is given by Lemma \ref{le:V-matrix-bound}.
We now formally state the asymptotic distribution of $\hat\theta$.

\begin{theorem}[{\bf Central limit theorem}]
\label{th:CL}
Under Conditions \ref{condi-t-r} and \ref{condi-pr}, if $\kappa=O(1)$, then, for any $s\ge2$, when $r\rightarrow \infty$, vector $\left(\hat\theta_2 - {\theta_2^*},\ldots,\hat\theta_s - {\theta_s^*}\right)$ asymptotically follows a multivariate normal distribution
with zero mean and  covariance matrix $\Sigma=(\sigma_{ij})$, where
$\sigma_{ij}=\delta_{ij}/v_{ii} + 1/v_{11}$.
\end{theorem}

\begin{remark}
\label{remark-b}
The asymptotic variance of $\hat{\theta}_i$ is $1/v_{ii} + 1/v_{11}$ for $i\ge 2$, which is in the order of $1/(tp)$ for $i=2,\ldots,r$ and $1/(rp)$ for $i=r+1,\ldots, r+t$
when all parameters are constant.
Under a different model identification condition $\sum_{i=1}^r \alpha_i=0$,  \cite{JMLR-Chen-2023} shows the asymptotic variance of $\hat{\theta}_i$ is $1/v_{ii}$.
The asymptotic variance of the MLE in \cite{haberman1977maximum} is implicit and involved with computing a supremum of a function.
\end{remark}

\section{Numerical studies}
\label{sec-simulation}

Here, we present numerical experiments conducted to
evaluate the performance of the MLE in the sparse Rasch model.

\subsection{Simulations}

We first investigate the errors in the MLEs in terms of the $\ell_\infty$-norm distance.
In the Erd\H{o}s\textendash R\'enyi random response design, we choose three values for $p$:
$p_1(s) = s^{-1/5}$, $p_2(s)=s^{-1/3}$, $p_3(s) = \frac{2\log s}{ s}$,
to measure the different sparsity levels, where $s$ may be $r$ or $t$.
For the individual merit parameters $\alpha_i$, $i=1,\ldots,r$ and item difficulty parameters $\beta_j$, $j=1,\ldots,t$, we set
$\alpha_i \sim U(-1/2, 1/2)$, and $\beta_j \sim N(0, 1/4)$.
To ensure model identification, we replace $\alpha_i$ with $\alpha_i - \alpha_1$
and $\beta_j$ with $\beta_j - \alpha_1$.

To observe how the error bound of the MLE changes with $r$ or $t$, we consider two scenarios:
(1) $t=1500$ and $r=500, 1000, 1500, 2000, 2500$; and (2) $r=1000$, $t=1000, 1500, 2000, 2500$, $3000$.
We repeat $50$ simulations for each combination $(r,t)$, and record three average values of
$\|\hat{\theta} - \theta\|_\infty$, $\|\hat{\alpha} - \alpha\|_\infty$ and $\|\hat{\beta} - \beta \|_\infty$.
The simulation results are presented in Figure \ref{fig:error}.

\begin{figure}[!htb]
\centering
\subfigure[$\ell_\infty$-norm errors under a fixed $t=1500$ and a varying $r$]{\includegraphics[width=0.9\textwidth, height=6.5cm]{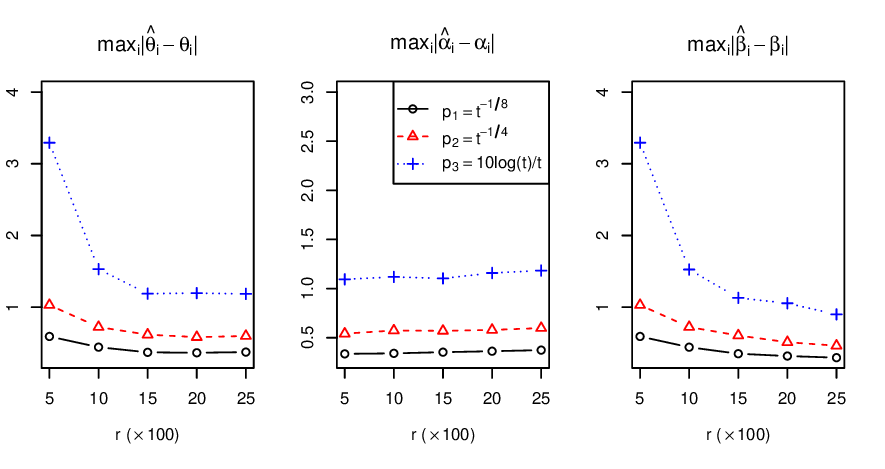}}
\subfigure[$\ell_\infty$-norm errors under a fixed $r=1000$ and a varying $t$]{\includegraphics[width=0.9\textwidth, height=6.5cm]{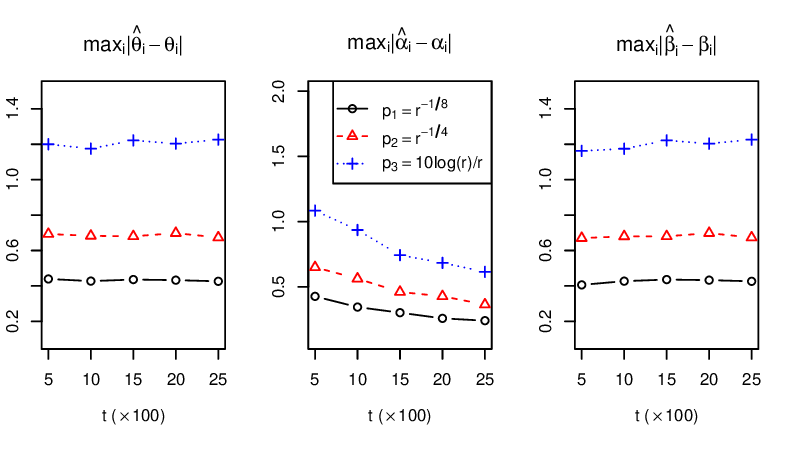}}
\caption{
The plots are the errors of the MLEs under different $p$, $r$, and $t$.
The horizontal axis shows the number of individuals $r$ in the first row or the number of items $t$ in the second row.
The vertical axis in each plot shows the average values of $\|\hat{\theta} - \theta \|_\infty$ in the first column,
$\|\hat{\alpha}  - \alpha \|_\infty$ in the second column,
and $\|\hat{\beta} - \beta \|_\infty$ in the third column.
}
\label{fig:error}     
\end{figure}

From Figure \ref{fig:error}, we observe that the errors increase when $p$ decreases when
for a fixed combination $(r,t)$.
Figure \ref{fig:error} (a) shows the error trend when $r$ varies and $t$ is fixed.
The values of $\|\hat{\theta} - \theta \|_\infty$ and $\|\hat{\alpha}  - \alpha \|_\infty$ decrease when $r$ increases.
When $p=t^{-1/8}$ and $p=t^{-1/4}$, the errors are close to zero when $r$ increases to $2500$.
However, the error in $\|\hat{\beta} - \beta \|_\infty$ does not exhibit an obvious decreasing trend as $r$ increases.
This may be because $t$ is maintained at a fixed value.
A similar phenomenon is observed in the plots of $\|\hat{\alpha} - \alpha\|_\infty$ in Figure \ref{fig:error} (b),
where the number of individuals $r$ is fixed.
Figure \ref{fig:error} (b) shows that the values of $\|\hat{\theta} - \theta \|_\infty$ and $\|\hat{\beta}  - \beta \|_\infty$ decrease as $t$ increases.

Next, we evaluate the asymptotic normality of the MLE.
The parameter values for $\alpha$, $\beta$ and $p$ are set as previously described.
According to the central limit theorem of MLE, Theorem \ref{th:CL}, both
\[
\xi_{i,j} := \frac{ [\hat{\alpha}_i-\hat{\alpha}_j-(\alpha_i^*-\alpha_j^*)]}{ \sqrt{
1/\hat{v}_{i,i}+1/\hat{v}_{j,j} }  }, ~~ i,j\in [r]
\]
and
\[
\eta_{i,j} := \frac{ [\hat{\beta}_i-\hat{\beta}_j-(\beta_i^*-\beta_j^*)]}{ \sqrt{
1/\hat{v}_{r+i, r+i}+1/\hat{v}_{r+j,r+j} } }, ~~ i,j\in [t]
\]
converge to the standard normal distribution, where $\hat{v}_{i,i}$ is the estimate of $v_{i,i}$ obtained by replacing $\theta^*$ with $\hat{\theta}$.
Recall that we set $\alpha_1=0$ for model identification.
We choose several pairs $(i,j)$: $(i,j)=(2,3), (r/2,r/2+1), (r-1,r)$ for $\xi_{i,j}$ and
 $(i,j)=(1,2), (t/2,t/2+1), (t-1,t)$ for $\eta_{i,j}$.
We report $95\%$ coverage frequencies for $\alpha_i - \alpha_j$ and $\beta_i-\beta_j$ as well as their confidence interval lengths.
We also draw quantile--quantile plots for $\xi_{i,j}$ and $\eta_{i,j}$. Each simulation is repeated $1000$ times.

\begin{figure}[!htpb]
\centering{
\includegraphics[width=0.9\textwidth, height=3in]{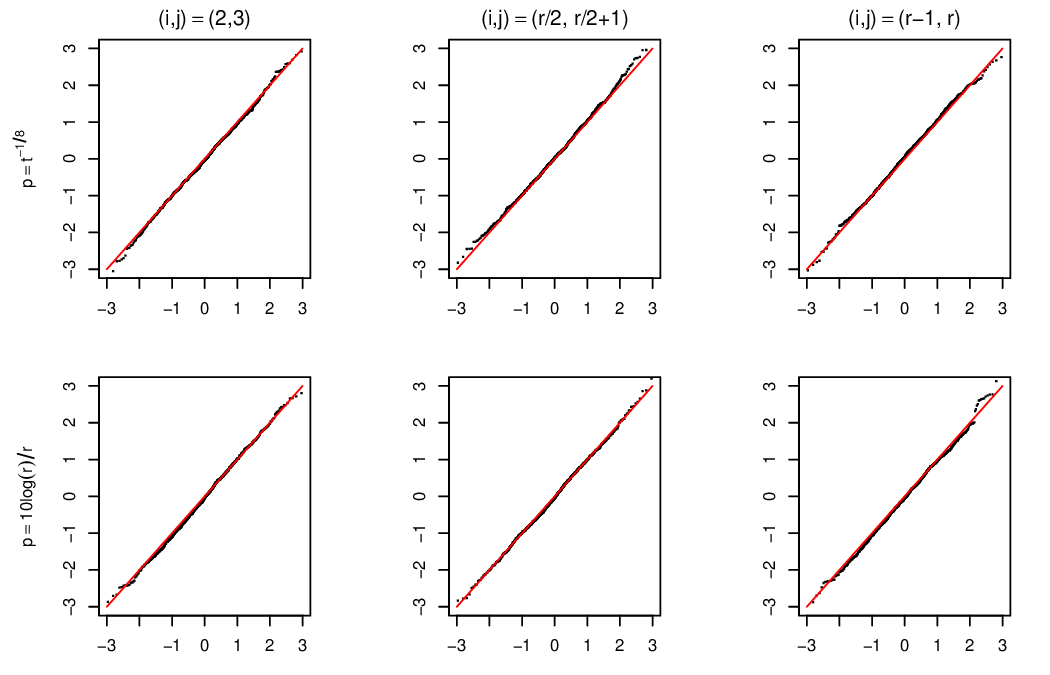}
}
\caption{ QQ-plots of $\xi_{i,j}$ under $(r,t)=(1000,1000)$.}
\label{figure-a}
\end{figure}

Figure \ref{figure-a} shows the empirical quantiles of $\xi_{i,j}$ against the theoretical quantiles of standard normality when $r=t=1000$, where the first and second rows correspond to  $p=t^{-1/8}$ and $p=10\log t/t$, respectively.
In this figure, the horizontal and vertical axes represent the quantile of standard normality and
the empirical quantile. The red line corresponds to baseline $y=x$
Most of the quantile-quantile plots of $\xi_{i,j}$ are close to the red diagonal line,
except for a slight derivation at the tail of the plots in the bottom-right subfigure.
This indicates that the asymptotic approximation is accurate. The plots for $\eta_{i,j}$ are similar and are not shown to save space.

Coverage frequencies of the $95\%$ confidence interval and the length of the confidence interval
are presented in Table \ref{tab1}. This table shows that most coverage frequencies are close to the target level of $95\%$.
This means that the asymptotic approximations are still very good, even when $p$ is very small. Note that, when $p=10\log t/t$, the network density
of the bipartite graph $\mathcal{G}_{r,t}$ is approximately $0.06$ when $r=t=1000$.
However, when $r$ and $t$ are fixed, the length of the confidence interval increases as $p$ decreases, as expected.
This is because fewer observations exist when $p$ decreases.

\begin{table}[!htbp]
\begin{center}
\centering
\setlength\tabcolsep{5pt}
\renewcommand{\arraystretch}{1.25}
\begin{footnotesize}
\begin{tabular}{l ccc ccc ccc}
  \hline
  \vspace{0.5mm}
&     && \multicolumn{3}{c}{ pair $(i,j)$ for $\xi_{i,j}$ }       && \multicolumn{3}{c}{ pair $(i,j)$ for $\eta_{i,j}$ }
     \\
      \cline{4-6}                              \cline{8-10}
$(r,t)$ & $p$  &&   $(2,3)$ & $(r/2, r/2+1)$ & $(r-1, r)$ & & $(1,2)$  & $(t/2,t/2+1)$  & $(t-1,t)$
\\
$(1000,1000)$  & $t^{-1/8}$   &&$94.6 ( 0.29 ) $&$ 94.9 ( 0.29 ) $&$ 95 ( 0.29 ) $&&$ 95.1 ( 0.291 ) $&$ 94.6 ( 0.291 ) $&$ 94 ( 0.291 ) $ \\
     & $t^{-1/4}$ &&$94 ( 0.449 ) $&$ 93.9 ( 0.448 ) $&$ 96.1 ( 0.448 ) $&&$ 94.7 ( 0.451 ) $&$ 95.7 ( 0.449 ) $&$ 94.2 ( 0.45 ) $ \\
     & $\frac{10\log t}{t} $&&$95 ( 0.727 ) $&$ 95 ( 0.728 ) $&$ 94.7 ( 0.729 ) $&&$ 94.9 ( 0.732 ) $&$ 95 ( 0.734 ) $&$ 93.3 ( 0.733 ) $ \\
  \vspace{0.5mm}
$(1000,1500)$  & $t^{-1/8}$   && $94.9 ( 0.242 ) $&$ 96.4 ( 0.242 ) $&$ 93.4 ( 0.242 ) $&&$ 95.8 ( 0.298 ) $&$ 95 ( 0.298 ) $&$ 94.2 ( 0.297 ) $ \\
     & $t^{-1/4}$ && $93.4 ( 0.385 ) $&$ 94.7 ( 0.385 ) $&$ 94 ( 0.385 ) $&&$ 94.7 ( 0.474 ) $&$ 94.4 ( 0.474 ) $&$ 94.7 ( 0.472 ) $ \\
     & $\frac{10\log t}{t} $&& $95.2 ( 0.71 ) $&$ 93.3 ( 0.709 ) $&$ 94.5 ( 0.711 ) $&&$ 95.8 ( 0.879 ) $&$ 94.3 ( 0.877 ) $&$ 95.3 ( 0.876 ) $ \\
    \hline
\end{tabular}
\end{footnotesize}
\end{center}
\caption{The reported values are the coverage frequencies ($\times 100\%$) of the $95\%$ confidence interval for a pair $(i,j)$ with
 the length of the confidence interval between parentheses.}
\label{tab1}
\end{table}

\subsection{Real data analysis}
\label{sec:real-data}

We use the public Riiid AIEd data set \citep{riiid-test-answer-prediction} as an illustrating application, available at \url{https://www.kaggle.com/c/riiid-test-answer-prediction/data}.
As mentioned before, this data set contains a very large number of individuals and items.
Actually, the size of the full data set is over $5$GB.
Because of computational cost and large required memory space for big response matrices, it is very time-consuming to compute the MLE for the full data set.
So we sampled randomly $1000$ items and then selected those individuals giving responses to at $300$ items amongst the $1000$ items.
A total number of $5485$ individuals was selected.
We treat the response matrix from an Erd\H{o}s\textendash R\'enyi random graph, where the sampling probability is $59\%$.
The minimum and maximum values of degrees for individuals are $89$ and $1467$, respectively, whereas those for items are $96$ and $3990$, respectively.
The estimates for $\hat{\alpha}_i$ and $\hat{\beta}_j$ and their estimated standard errors are listed in Table \ref{tab2}, where only the top $5$ 
individuals and items are selected for presentation because of an excess of the estimated nodes.
Here, we set the estimate for the ability parameter of individual ``1509564249" with the minimum estimate to zero.
The estimates of parameters for people range from a minimum $0$ to a maximum $4.39$
whereas those of the parameters for items range from a minimum $0.41$ to a maximum $1.67$.
We plot the histogram for $\hat{\alpha}_i$ and $\hat{\beta}_j$ in Figure \ref{figure-b}.
As we can see, there are evident differences for the ability of individuals and the difficulty of items.
We conduct a Wald-type statistic for testing whether there is significant difference amongst the top $5$ individuals or $5$ items in Table \ref{tab2}.
The results show that there are no significant difference for the top $5$ individuals or items.

\begin{figure}[!htpb]
\centering{
\includegraphics[width=0.9\textwidth, height=2in]{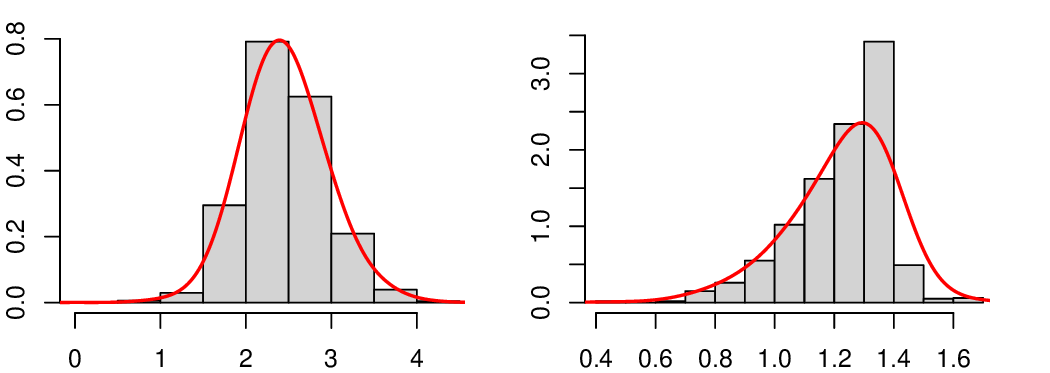}
}
\caption{The histograms for the MLEs $\hat{\alpha}_i$ in the left and $\hat{\beta}_j$ in the right. The red line is the density curve. }
\label{figure-b}
\end{figure}

\begin{table}[!htpb]
\begin{center}
\caption{Estimators of $\hat{\alpha}_i$, $\hat{\beta}_j$ and their standard errors.}
\label{tab2}
\centering
\setlength\tabcolsep{5pt}
\renewcommand{\arraystretch}{1.25}
\begin{footnotesize}
\begin{tabular}{cccc c cccc}
\hline
  {People ID} & {Degree} & {$\hat{\alpha}_i$} & {$\hat{\sigma}_i$} & & {Item ID} & {Degree} & {$\hat{\beta}_j$} & {$\hat{\sigma}_j$}
   \vspace{1mm}  \\
\hline
$524304795 $&$ 610 $&$ 4.39 $&$ 0.21 $&&$ 33 $&$ 3655 $&$ 1.67 $&$ 0.08$ \\
$1223856469 $&$ 942 $&$ 4.23 $&$ 0.17 $&&$ 31 $&$ 3650 $&$ 1.67 $&$ 0.08$ \\
$581787196 $&$ 121 $&$ 4.2 $&$ 0.39 $&&$ 79 $&$ 3666 $&$ 1.63 $&$ 0.08$ \\
$2056466762 $&$ 121 $&$ 4.2 $&$ 0.39 $&&$ 50 $&$ 3672 $&$ 1.62 $&$ 0.08$ \\
$1045247032 $&$ 951 $&$ 4.19 $&$ 0.16 $&&$ 58 $&$ 3700 $&$ 1.62 $&$ 0.08$
 \\
\hline
\end{tabular}
\end{footnotesize}
\end{center}
\end{table}

\section{Discussion}
\label{sec-discussion}

We studied the statistical properties of the well-known Rasch model in sparse-item response experiments under the
Erd\H{o}s\textendash R\'enyi random sampling design. In particular, we established the uniform consistency of the MLE
and obtained its asymptotic normality 
when both the number of individuals $r$ and number of items $t$ go to infinity.
The sampling probability $p$ in the
Erd\H{o}s\textendash R\'enyi graph can be as small as  $\log r/r$ up to a constant factor, where the MLE is still consistent.
This is a fundamental requirement to guarantee the existence of the MLE. 

Several issues need to be addressed in future studies.
First, we assume $c_1\le r/t\le c_2$ for the two positive constants $c_1$ and $c_2$.
In other applications, the number of items may be relatively small, such that $r/t \to \infty$.
In this case, the degrees of the individuals were far smaller than those of the items,
which leads to very different upper bounds of $d_i$ and other quantities in the proofs.
It is unclear whether the current proof strategies can be extended.
However, investigating this issue is beyond the scope of the present study.
Second, we assumed that the individual ability parameter $\alpha_i$ and item difficulty parameter $\beta_j$
are bounded above by a constant. To determine whether this assumption can be relaxed, we performed some additional simulations.
The results showed that the MLE performs well for a wide range of parameters.
Furthermore, the asymptotic behaviour of the estimator depends not only on $\kappa$ but also on the configuration of the parameters.
It will be interesting to investigate this issue in future studies.

\section{Appendix}
\label{sec-appendix}

\subsection{Preliminaries}
\label{subsec-pre}

The gradient of negative log-likelihood function ${\ell}(\theta  )$ defined in \eqref{negative-likelihood-fu}, can be written as:
\begin{align}\label{eq:gradient}
\nabla {\ell}(\theta )=\sum\limits_{i \in [r]} {\frac{{\partial \ell (\theta )}}{{\partial {\theta _i}}}}\cdot {e_i}  + \sum\limits_{j\in [t]} {\frac{{\partial \ell (\theta )}}{{\partial {\theta _{j+r}}}}}\cdot {e_{j+r}},
\end{align}
where $e_1, \ldots, e_{r+t}$ represent the standard basis vectors in ${\mathbb{R}^{r+t}}$.
In the above equation:
\begin{eqnarray*}
\label{eq:gradient-i}
\frac{{\partial \ell (\theta )}}{{\partial {\theta _i}}} & = & \sum\limits_{(i,j) \in {\mathcal E},j\in [t]} {\left. {\left\{ {\mu ({\theta _i} - {\theta_{j+r}}) - {a_{i,j+r}}} \right.} \right\}},~~~~i \in [r],
\\
\label{eq:gradient-j}
\frac{{\partial \ell (\theta )}}{{\partial {\theta _{j+r}}}} & = &
-\sum\limits_{(i,j) \in {\mathcal E},i \in [r]} {\left. {\left\{ {\mu ({\theta _i} - {\theta_{j+r}}) - {a_{i,j+r}}} \right.} \right\}},~~j\in [t].
\end{eqnarray*}
Furthermore, the Hessian matrix of ${\ell}(\theta)$ can be written as:
\begin{align}
\label{eq:Hessian}
{\nabla ^2}{\ell}(\theta ) = \sum\limits_{(i,j) \in {\mathcal E},i \in [r],j\in [t]} {\mu^\prime({\theta_i-\theta_{j+r}})\left( {{e_i} - {e_{j+r}}} \right){{\left( {{e_i} - {e_{j+r}}} \right)}^\top }},
\end{align}
where $\mu^\prime(x) = e^x/( 1 + e^x )^2$.
Because $\mu^\prime(x)$ is a symmetric function, we have:
\[
\frac{ e^{ \alpha_i - \beta_j } }{ (1 + e^{\alpha_i - \beta_j } )^2 } = \frac{ e^{-|\alpha_i-\beta_j|} }{ (1 + e^{-|\alpha_i-\beta_j|} )^2 }
\ge \frac{1}{4} e^{-|\alpha_i-\beta_j|}.
\]
It follows that, when $\| \theta - \theta^* \|_\infty \le C$,  for all $i\in [r], j\in [t]$, we have:
\begin{equation}
\label{ineq-mu-bn-cn}
\frac{1}{4\kappa e^{2C} } \le \mu^\prime({\theta_i-\theta_{j+r}}) \le \frac{1}{4},
\end{equation}
by noticing that $x(1-x) \le \frac{1}{4}$ when $ x\in (0, 1)$
and
\[
|\alpha_i - \beta_j | \le \theta_{\max}^* - \theta_{\min}^* + 2 \|\theta - \theta^* \|_\infty \le \log \kappa + 2C.
\]

For convenience, we transform parameter $\theta_i$ to $\omega_i=\theta_i-\bar{\theta}$ for each $i\in [r+t]$,
where $\bar{\theta}$ is the mean of all $\theta_i$; that is, $\bar{\theta} = (\sum_i \theta_i )/(r+t)$.
We denote $\omega=(\omega_1, \ldots, \omega_{r+t})^\top$.
Observe that
\begin{equation}
\label{eq:keshibie}
\omega_i-\omega_{j+r}=\theta_i-\theta_{j+r}~~~\text{and}~~~\sum_{i\in [r+t]}{ \omega_i}=0.
\end{equation}
and the likelihood function is made invariable by replacing $\theta$ with $\omega$.
Then, we define:
\begin{equation}
\label{eq:hat-omega}
\hat{\omega }={\mathop{\arg \min }}_{\omega \in \mathbb{R}^{r+t} }\,{\ell  }(\omega ).
\end{equation}

To establish the existence and consistency of $\hat{\omega}$, we present some lemmas below, whose proofs are provided in the supplementary material.

\begin{lemma}\label{le:A0}
Let $d_i$ be the degree of node $i$ in the bipartite Erd\H{o}s\textendash R\'enyi random graph $G_{r,t}$, $d_{\min} = \min_{i \in [r+t]} d_i $ and $d_{\max} =  \max_{i \in [r+t]} d_i $.
Suppose that conditions \ref{condi-t-r} and \ref{condi-pr} hold true. Then, for sufficiently large $c_0$, the event
\begin{align} \label{eq:A0}
\mathcal{A}_0:= \left\{ {\left. {\frac{{rp}}{2} \le {d_{{\rm{min}}}} \le {d_{{\rm{max}}}} \le \frac{{3tp}}{2}} \right\}} \right.
\end{align}
holds with a probability of at least $1 - O({r^{ - 10}})$.
\end{lemma}

\begin{lemma}\label{le:min-max-eigenvalue}
Under Conditions \ref{condi-t-r} and \ref{condi-pr}, for a sufficiently large constant $c_0>0$, we have:
\begin{equation*}
P\bigg(\lambda_{\max }\left( {{\nabla }^{2}}{{\ell }}(\omega^* ) \right)\le\frac{3}{4}tp\bigg)\ge 1-O(r^{-10}),
\end{equation*}
where $\lambda_{\max}(A)$ is the largest eigenvalue of $A$.
\end{lemma}

\begin{lemma}\label{le:second-min-eigenvalue}
Let
\[
\lambda_{\min, \perp}(A) = \min\{ \mu | z^\top A z \ge \mu \|z\|_2^2 \mbox{~for all $z$ with $\mathbf{1}^\top z =0$} \}.
\]
That is, it has the smallest eigenvalue when restricted to vectors orthogonal to $\mathbf{1}$.
Under Conditions \ref{condi-t-r} and \ref{condi-pr}, for a sufficiently large constant $c_0>0$, we have:
\begin{equation*}
P\bigg({\lambda }_{\min, \perp} \left( {{\nabla }^{2}}{{\ell }}( \omega^* ) \right)\ge \frac{1}{ 4\kappa }rp\bigg)\ge 1-O(r^{-10}).
\end{equation*}
\end{lemma}

\begin{lemma}\label{le:ball12}
With a probability of at least $1-O(r^{-10})$, we obtain:
\begin{align}
\label{le:grad-L2}
& \|  \nabla \ell ({\omega ^*}) \|_2^2 \le C rtp\log r,
\\
\label{le:ball12-1}
& \mathop {\max }\limits_{i \in [r]} {\Big[ {\sum\limits_{(i,j) \in {\mathcal E},j\in [t] } {(-{a_{i,j+r}} + \mu (\omega_i^*-\omega_{j+r}^*))} } \Big]^2}
\le C tp\log r,
\\
\label{le:ball12-2}
& \mathop {\max }\limits_{j\in [t]} {\Big[ {\sum\limits_{(i,j) \in {\mathcal E},i \in [r] } {(-{a_{i,j+r}} + \mu (\omega_i^*-\omega_{j+r}^*))} } \Big]^2} \le C rp\log r,
\\
\label{le:ball12-3}
& \mathop {\max }\limits_{i \in [r]} \sum\limits_{(i,j) \in {\mathcal E},j\in [t]} {{{( - {a_{i,j+r}} + \mu (\omega _i^* - \omega_{j+r}^*))}^2}} \le C tp\log r,
\\
\label{le:ball12-4}
&\mathop {\max }\limits_{j\in [t]} \sum\limits_{(i,j) \in {\mathcal E},i \in [r]} {{{(-{a_{i,j+r}} + \mu (\omega _i^* - \omega_{j+r}^*))}^2}}   \le C rp\log r,
\end{align}
where $C$ is  a sufficiently large constant. 
\end{lemma}

\begin{lemma}\label{le:ball-p}
Under Conditions \ref{condi-t-r} and \ref{condi-pr},  with a probability of at least $1 - O({r^{ - 10}})$, it holds that:
\begin{small}
\begin{eqnarray}
\label{eq:ball-p1}
\max_{i\in[r]}\sum_{j\in[t]}(X_{i,j+r}-p)^2 & \leq & C^\prime tp,
\\
\label{eq:ball-p2}
\max_{j\in[t]}\sum_{i\in[r]}(X_{i,j+r}-p)^2 & \leq  & C^\prime rp,
\end{eqnarray}
\end{small}
where $C^\prime$ is a sufficiently large constant.
\end{lemma}

\subsection{The existence of the MLE}
\label{subsec-exis-MLE}

Here, we aim to show that the MLE $\hat\omega$ exists and is bounded by a high probability.
We define the regularised MLE ${\hat \omega _\lambda }$ as follows:
\begin{equation}\label{eq:argmin_omega}
{\hat \omega _\lambda }: = {\mathop {\arg \min }\limits_{\omega  \in {^{r + t}}} {\mkern 1mu} \Big\{ \ell (\omega ) + \frac{1}{2}\lambda {{\left\| \omega  \right\|}_2}} \Big\},
\end{equation}
where $\lambda>0$ is a tuning parameter.
We use it as a middleman to demonstrate the existence of the MLE.
To analyse the properties of $\hat{\omega}_\lambda$, we use the following gradient descent sequence:
\begin{equation}
\label{eq:gradient-update}
\omega_{\lambda}^{(k+1)}={\omega}_{\lambda}^{(k)} - \eta({\nabla \ell ({\omega_{\lambda}^{(k)}})} + \lambda{\omega}_{\lambda}^{(k)}),~~k=0,1,\ldots
\end{equation}
where the initial point is set to $\omega^*$ and we adopt the time-invariant step rule as follows:
\[
\eta=\frac{1}{\lambda+tp},
\]
where $\lambda = 1/(r + t)$.
Our proof strategy is motivated by \cite{chen2019spectral} and \cite{chen2022partial},
who used the leave-one-out method to prove the consistency of MLE in the Bradley--Terry model.
Here, we briefly describe this concept.
Note that, if $\hat{\omega}$ exists and is bounded, then:
\[
\hat{\omega} - \omega^* = (\hat{\omega}-\hat{\omega}_{\lambda}) + (\hat{\omega}_\lambda - \omega_\lambda^{(k)}) +
(\omega_\lambda^{(k)} - \omega^*).
\]
Thus, the proof proceeds in three steps. First, we use the leave-one-out argument to demonstrate that output $\omega_\lambda^{(t)}$
is close to truth $\omega^{*}$ in an entrywise:
\begin{equation}
\label{lemma-omega-lambda}
\| \omega_\lambda^{(k)}-\omega^{*} \|_{\infty}  \leq 2.
\end{equation}
Second, we use standard optimisation theory to show that the output $\omega^{(t)}$  is close to the regularised MLE $\hat{\omega}_\lambda$:
\begin{eqnarray}
\label{le:stronglyconvex}
\| {\omega}^{(k)}_\lambda-{\hat\omega_\lambda} \|_{2} & \leq & \left(1-\frac{\lambda}{\lambda+tp}\right)^{T}\left\Vert {\hat\omega_\lambda}-{\omega^*}\right\Vert _{2},
\\
\label{le:entrywise-hat-MLE}
{\left\| {{{\hat \omega }_\lambda } - {\omega ^*}} \right\|_\infty } & \leq & 4.
\end{eqnarray}
The proofs of \eqref{lemma-omega-lambda}, \eqref{le:stronglyconvex}, and \eqref{le:entrywise-hat-MLE} are provided in the Supplementary Material.
Third, we show that $\hat\omega_\lambda$ is the interior of a compact set by bridging the regularised MLE and restricted MLE.
We can now formally state the following result.

\begin{lemma}
\label{lemma-existence}
Under Conditions \ref{condi-t-r} and \ref{condi-pr}, if $\kappa=O(1)$, then $\hat{\omega}$ exists and satisfies:
\begin{align}
\label{eq-lemma-ext}
\left\Vert \hat\omega-\omega^{*}\right\Vert _{\infty}\leq 5,
\end{align}
with the probability of at least $1-O(t^3r^{-10})$.
\end{lemma}

\begin{proof}[Proof of Lemma \ref{lemma-existence}]
We define a restricted MLE as:
\begin{align}
\label{eq:con-MLE}
\widetilde {\omega}:=\mathop {\arg \min }\limits_{\mathbbm{1}_{r + t}^ \top \omega  = 0:{{\left\| {\omega  - {\omega ^*}} \right\|}_\infty } \le 5} \ell (\omega ).
\end{align}
From \eqref{le:entrywise-hat-MLE}, $\hat {\omega}_{\lambda}$ satisfies the constraint in \eqref{eq:con-MLE}.
Therefore, we have:
\begin{align*}
\ell(\hat {\omega}_{\lambda}) \geq \ell(\widetilde {\omega}).
\end{align*}
Applying a second-order Taylor expansion yields:
\begin{align*}
\ell ({\widetilde\omega }) = \ell ({\hat \omega _\lambda }) + {({\widetilde \omega } - {\hat \omega _\lambda })^ \top }\nabla \ell ({\hat \omega _\lambda }) + \frac{1}{2}{({\widetilde\omega } - {\hat \omega _\lambda })^ \top }{\nabla ^2}\ell (\tau )({\widetilde \omega } - {\hat \omega _\lambda }),
\end{align*}
where $\tau$ is a convex combination of $\widetilde\omega $ and $\hat \omega _\lambda$.
From \eqref{le:entrywise-hat-MLE}, $\|\hat {\omega}_{\lambda}-\omega^*\|_{\infty}\leq 4$.
Owing to constraint $\|\widetilde{\omega}-\omega^*\|_{\infty}\leq 5$, we have $\|\tau-\omega^*\|_{\infty}\leq 5$.
As $\ell ({\widetilde\omega })\le \ell ({\hat \omega _\lambda })$, we also have:
\begin{align*}
\frac{1}{2}{({{\widetilde\omega }} - {{\hat \omega }_\lambda })^ \top }{\nabla ^2}\ell (\tau )({{\widetilde \omega }} - {{\hat \omega }_\lambda })
& \le  - {({{\widetilde \omega }} - {{\hat \omega }_\lambda })^ \top }\nabla \ell ({{\hat \omega }_\lambda })
\end{align*}
Using the Cauchy--Schwarz inequality, we obtain:
\begin{align*}
- {({{\widetilde \omega }} - {{\hat \omega }_\lambda })^ \top }\nabla \ell ({{\hat \omega }_\lambda }) \le {\left\| {\nabla \ell ({{\hat \omega }_\lambda })} \right\|_2}{\left\| {{{\widetilde \omega }} - {{\hat \omega }_\lambda }} \right\|_2}.
\end{align*}
Combining these yields:
\[
\frac{1}{2}\lambda_{\min,\perp}( \nabla^2 \ell(\tau ) ) \| ({{\widetilde \omega }} - {{\hat \omega }_\lambda }) \|_2^2
\le {\left\| {\nabla \ell ({{\hat \omega }_\lambda })} \right\|_2}{\left\| {{{\widetilde \omega }} - {{\hat \omega }_\lambda }} \right\|_2}.
\]
From Lemma \ref{le:second-min-eigenvalue}, it follows that:
\begin{align*}
{\left\| {{{\widetilde \omega }} - {{\hat \omega }_\lambda }} \right\|_2} \lesssim \frac{{{{\left\| {\nabla \ell_\lambda ({{\hat \omega } })} \right\|}_2}}}{ rp/(4\kappa) }.
\end{align*}
Given that $\nabla\ell(\hat {\omega}_{\lambda})+\lambda\hat {\omega}_{\lambda}=0$, we have:
\begin{align*}
\left\| {{{\widetilde \omega }} - {{\hat \omega }_\lambda }} \right\|_2
\lesssim \frac{ \lambda \| {{{\hat \omega }_\lambda }}\|_2 }{ rp/(4\kappa) }
\lesssim \frac{ \lambda (r+t)^{1/2} \| {{{\hat \omega }_\lambda }}\|_2 }{ rp/(4\kappa) }
\lesssim \frac{  \kappa \|\omega^*\|_\infty }{ rp (r+t)^{1/2} },
\end{align*}
where we set $\lambda = (r+t)^{-1}$.
Therefore, for a sufficiently large $r$, if $\kappa=O(1)$ and $ \kappa \|\omega^*\|_\infty/\{ (rp) (r+t)^{1/2}\} \to 0$,
then:
\begin{align*}
\|\widetilde {\omega}-\omega^*\|_{\infty} \leq \|\hat {\omega}_{\lambda}-\omega^*\|_{\infty}+\|\widetilde {\omega}-\hat {\omega}_{\lambda}\|_2
\leq 4 + \|\widetilde {\omega}-\hat {\omega}_{\lambda}\|_2 < 5.
\end{align*}
The minimizer in \eqref{eq:con-MLE} is in the interior of the constraint. From the convexity of \eqref{eq:con-MLE}, we have $\widetilde {\omega}=\hat {\omega}$ such that \eqref{eq-lemma-ext} holds.
This completes the proof.
\end{proof}

\subsection{Proof of Theorem \ref{th:MLE-main}}
\label{sub-consistency-proof}

Theorem \ref{th:MLE-main} contains two claims, \eqref{eq:theorem-con-a} and \eqref{eq:theorem-con-b}:
The proofs are presented in Sections \ref{subsub-th1-a} and \ref{subsub-th1-b}, respectively.

\subsubsection{Proof of \eqref{eq:theorem-con-a}}
\label{subsub-th1-a}

We aim to demonstrate that the MLE $\hat{\omega}$ is close to the true $\omega^*$ in terms of the $\ell_2$-norm distance.
From the definition of $\hat{\omega}$ in \eqref{eq:hat-omega}, we have $\ell(\omega^*)\geq\ell(\hat{\omega})$.
We apply a second-order Taylor expansion to $\ell(\hat{\omega})$ to obtain:
\begin{align*}
\ell(\hat{\omega})=\ell(\omega^*) + (\hat{\omega}-\omega^*)^T\nabla\ell(\omega^*) + \frac{1}{2}(\hat{\omega}-\omega^*)^\top \nabla^2\ell(\tilde \omega)(\hat{\omega}-\omega^*),
\end{align*}
where $\tilde \omega$ denotes a convex combination of $\hat\omega$ and $\omega^*$.
From Lemma \ref{lemma-existence}, $\|\hat{\omega}-\omega^*\|_{\infty}\leq 5$, which implies $\|\tilde \omega-\omega^*\|_{\infty}\leq 5$.
Thus, from Lemma \ref{le:second-min-eigenvalue}, we have:
\begin{align*}
\frac{1}{2}(\hat{\omega}-\omega^*)^\top \nabla^2\ell(\tilde \omega)(\hat{\omega}-\omega^*)\geq \frac{1}{8\kappa}rp\|\hat{\omega}-\omega^*\|_2^2.
\end{align*}
Therefore, together with $\ell(\omega^*)\geq\ell(\hat{\omega})$, we have:
\begin{align*}
\ell(\omega^*)\geq\ell(\hat{\omega})=\ell(\omega^*) + (\hat{\omega}-\omega^*)^\top \nabla\ell(\omega^*) + \frac{1}{8\kappa}rp\|\hat{\omega}-\omega^*\|_2^2.
\end{align*}
Using the Cauchy-Schwarz inequality, we obtain
\begin{align*}
 \frac{1}{8\kappa}rp\|\hat{\omega}-\omega^*\|_2^2 \le -(\hat{\omega}-\omega^*)^\top \nabla\ell(\omega^*)\le \|\nabla\ell(\omega^*)\|_2\|\hat{\omega}-\omega^*\|_2.
 \end{align*}
From Lemma \ref{le:ball12} it follows that:
\begin{align*}
\|\hat{\omega}-\omega^*\|_2^2\leq \frac{\|\nabla\ell(\omega^*)\|_2^2}{\left(\frac{1}{8\kappa}rp\right)^2} \lesssim \frac{\kappa^2 \log r}{p}.
\end{align*}
This completes the proof.

\subsubsection{Proof of \eqref{eq:theorem-con-b}}
\label{subsub-th1-b}

We aim to demonstrate that the MLE $\hat{\omega}$ is close to the truth $\omega^*$ in terms of the $\ell_\infty$-norm distance.
Let $\omega_m\in\mathbb{R}$ denote the $m$-th entry of $\omega$ and $\omega_{-m}=(\omega_{1},\ldots,\omega_{m-1},\omega_{m+1},\ldots,\omega_{r+t})\in\mathbb{R}^{r+t-1}$ be the remaining entries.
We divide the negative log-likelihood function $\ell(\omega)$ into two components:
\begin{equation}\label{eq:MLE-new}
\ell(\omega) = {\ell}^{(-m)}(\omega_{-m}) + {\ell}^{(m)}(\omega_m|\omega_{-m}),
\end{equation}
where
\begin{equation}
\label{define-ell-m-drop}
\ell^{(-m)}( \omega_{-m} )
= \begin{cases}
\sum\limits_{ \begin{smallmatrix}i \in [r]\backslash \{ m\} \\ j \in [t] \end{smallmatrix} } X_{i,j + r}
\Big\{  {{a_{i,j + r}}\log \frac{1}{{\mu ({\omega _i} - {\omega _{j + r}})}}} +
                       {(1 - {a_{i,j + r}})\log \frac{1}{{1 - \mu ({\omega _i} - {\omega _{j + r}})}}} \Big\}, & m \in [r],
\\
 \sum\limits_{ \begin{smallmatrix}i \in [r] \\ j \in [t]\backslash \{ m-r\} \end{smallmatrix} } X_{i,j + r}
\Big\{  {{a_{i,j + r}}\log \frac{1}{{\mu ({\omega _i} - {\omega _{j + r}})}}} +
                       {(1 - {a_{i,j + r}})\log \frac{1}{{1 - \mu ({\omega _i} - {\omega _{j + r}})}}} \Big\}, & m-r \in [t]
\end{cases}
\end{equation}
and
\begin{equation}
\ell^{(m)}( \omega_m | \omega_{-m} )
= \begin{cases}
 \sum\limits_{j \in [t]} {{X_{m,j + r}}}\Big\{  {{a_{m,j + r}}\log \frac{1}{{\mu ({\omega _m} - {\omega _{j + r}})}}} +
{(1 - {a_{m,j + r}})\log \frac{1}{{1 - \mu ({\omega _m} - {\omega _{j + r}})}}} \Big\}, & m \in [r],
\\
\sum\limits_{i \in [r]} {{X_{i,m}}} \Big\{ {{a_{i,m}}\log \frac{1}{{\mu ({\omega _i} - {\omega _m})}}} +
  {(1 - {a_{i,m}})\log \frac{1}{{1 - \mu ({\omega_i} - {\omega _m})}}} \Big\}, & m-r \in [t].
\end{cases}
\end{equation}
We define
\begin{align}
\label{eq:con-MLE-l-m}
{\hat{\omega}}_{-m}^{(m)} := \mathop {\arg \min }\limits_{{\omega_{-m}}:{{\left\| {{\omega _{-m}} - \omega_{-m}^*} \right\|}_\infty } \le 5} {\ell^{(-m)}}({\omega_{-m}}).
\end{align}
We denote ${\varphi ^{(m)}}({\omega _m}|{\omega _{ - m}})$ by the first derivative of  ${\ell ^{(m)}}({\omega _m}|{\omega _{ - m}})$  with respect to $\omega _m$:
\begin{align}
{\varphi ^{(m)}}({\omega _m}|{\omega _{ - m}})
= \begin{cases}
\sum\limits_{j \in [t]} {{X_{m,j + r}}\left( { - {a_{m,j + r}} + \mu ({\omega _m} - {\omega _{j + r}})} \right)}, &  m \in [r],
\\
\sum\limits_{i \in [r]} {{X_{i,m}}\left( { - {a_{i,m}} + \mu ({\omega _i} - {\omega _m})} \right)}, & m - r \in [t].
\end{cases}
\end{align}
Denote ${\psi ^{(m)}}({\omega _m}|{\omega _{ - m}})$ by the second derivative of  ${\ell ^{(m)}}({\omega _m}|{\omega _{ - m}})$  with respect to $\omega _m$:
\begin{align}
\label{def-psi-m}
\psi^{(m)}( \omega_m | \omega_{-m} )
= \begin{cases}
 \sum\limits_{j \in [t]} {{X_{m,j + r}}\mu '({\omega _m} - {\omega _{j + r}})},  &  m \in [r],
 \\
 \sum\limits_{i \in [r]} {{X_{i,m}}\mu '({\omega _i} - {\omega _m})},  & m - r \in [t].
\end{cases}
\end{align}

Now, we present a lemma that can be used in the proofs.

\begin{lemma}
\label{le:-m-no-ave-a}
Define $a_m:=\mathsf{ave}( \hat{\omega}_{-m}^{(m)}-\omega_{-m}^*)$.
Under Conditions \ref{condi-t-r} and \ref{condi-pr},
as $r\rightarrow\infty$, we have:
\begin{align}
\max_{m\in[r+t]}\|\hat{\omega}_{-m}^{(m)}-\omega_{-m}^*-a_m\mathbbm{1}_{r+t-1}\|_2^2\lesssim\frac{\log r}{p},
\end{align}
with probabilities of at least $1-O(r^{-10})$.
\end{lemma}

\begin{proof}[Proof of Lemma \ref{le:-m-no-ave-a}]
When $c\in\mathbb{R}$, we have:
\begin{align*}
{\ell}^{(-m)}(\omega_{-m})&={\ell}^{(-m)}(\omega_{-m}+c\mathbbm{1}_{r+t-1}),\\
\nabla{\ell}^{(-m)}(\omega_{-m})&=\nabla{\ell}^{(-m)}(\omega_{-m}+c\mathbbm{1}_{r+t-1}), \\ \nabla^2{\ell}^{(-m)}(\omega_{-m})&=\nabla^2{\ell}^{(-m)}(\omega_{-m}+c\mathbbm{1}_{r+t-1}).
\end{align*}
According to the definition of ${\hat{\omega}}_{-m}^{(m)}$  in \eqref{eq:con-MLE-l-m} and using Taylor's expansion, we have:
\begin{equation*}
\begin{split}
{\ell}^{(-m)}(\omega_{-m}^*)& \geq  {\ell}^{(-m)}({\hat{\omega}}_{-m}^{(m)}) \\
&= {\ell}^{(-m)}(\omega_{-m}^*) + ({\hat{\omega}}_{-m}^{(m)}-\omega_{-m}^*-a_m\mathbbm{1}_{r+t-1})^\top \nabla{\ell}^{(-m)}(\omega_{-m}^*) \\
&\quad + \frac{1}{2}({\hat{\omega}}_{-m}^{(m)}-\omega_{-m}^*-a_m\mathbbm{1}_{r+t-1})^\top \nabla^2{\ell}^{(-m)}(\xi)({\hat{\omega}}_{-m}^{(m)}-\omega_{-m}^*-a_m\mathbbm{1}_{r+t-1}),
\end{split}
\end{equation*}
where $\xi=\omega_{-m}^*+\tau\left({\hat{\omega}}_{-m}^{(m)}-\omega_{-m}^*\right)$ and $0\le \tau \le 1$.
From the definition of ${\hat{\omega}}_{-m}^{(m)}$  in \eqref{eq:con-MLE-l-m}, we obtain:
$$\|\xi-\omega_{-m}^*\|_{\infty}\leq \|{\hat{\omega}}_{-m}^{(m)}-\omega_{-m}^*\|_{\infty}\leq 5.$$
By Lemma \ref{le:second-min-eigenvalue}, with probability $1-O ({r^{-10}})$, we have:
\begin{align*}
&\frac{1}{2}({\hat{\omega}}_{-m}^{(m)}-\omega_{-m}^*-a_m\mathbbm{1}_{r+t-1})^\top \nabla^2{\ell}^{(-m)}(\xi)({\hat{\omega}}_{-m}^{(m)}-\omega_{-m}^*-a_m\mathbbm{1}_{r+t-1})\\
 \geq & \frac{1}{4b_n}rp\|{\hat{\omega}}_{-m}^{(m)}-\omega_{-m}^*-a_m\mathbbm{1}_{r+t-1}\|_2^2.
\end{align*}
Using the Cauchy--Schwarz inequality, we obtain:
\begin{align*}
\|{\hat{\omega}}_{-m}^{(m)}-\omega_{-m}^*-a_m\mathbbm{1}_{r+t-1}\|_2\leq \frac{\|\nabla{\ell}^{(-m)}(\omega_{-m}^*)\|_2}{(\frac{1}{4b_n}rp)}.
\end{align*}
By combining Lemma \ref{le:grad-L2} with $1 - O({r^{ - 10}})$, we have:
$$\|{\hat{\omega}}_{-m}^{(m)}-\omega_{-m}^*-a_m\mathbbm{1}_{r+t-1}\|^2_2\lesssim \frac{\log r}{p}.$$
This completes the proof.
\end{proof}

Now, we can prove \eqref{eq:theorem-con-b}.

\begin{proof}[Proof of \eqref{eq:theorem-con-b}]

When we use the leave-one-out function defined in \eqref{eq:MLE-new} to prove  \eqref{eq:theorem-con-b}, we consider two cases for $m$.
(Case I) $m\in [r]$ and (Case II) $m\in \{r+1, \ldots, r+t\}$.
The proofs for the two cases are very similar. Therefore, we present only the proofs for Case I.
In the following, we assume $m\in [r]$.

From Lemma \ref{lemma-existence}, we have $\|\hat{\omega}_{-m}-\omega^*_{-m}\|_{\infty}\leq \|\hat{\omega}-\omega^*\|_{\infty}\leq 5$.
Therefore, $\hat{\omega}_{-m}$ satisfies the constraint in \eqref{eq:con-MLE-l-m}.
From the definition of ${\hat{\omega}}_{-m}^{(m)}$ in \eqref{eq:con-MLE-l-m}, we obtain:
\begin{align*}
\ell^{(-m)}(\hat{\omega}_{-m}) &\geq {\ell}^{(-m)}({\hat{\omega}}_{-m}^{(m)}) \\
&= \ell^{(-m)}(\hat{\omega}_{-m}) + (\hat{\omega}^{(m)}_{-m}-\hat{\omega}_{-m}-\bar{a}_m\mathbbm{1}_{r+t-1})^\top\nabla\ell_n^{(-m)}(\hat{\omega}_{-m}) \\
&+ \frac{1}{2}( \hat{\omega}^{(m)}_{-m}-\hat{\omega}_{-m}-\bar{a}_m\mathbbm{1}_{r+t-1})^\top \nabla^2\ell_n^{(-m)}(\xi)( \hat{\omega}^{(m)}_{-m}-\hat{\omega}_{-m}-\bar{a}_m\mathbbm{1}_{r+t-1}),
\end{align*}
where $\xi$ is a convex combination of $\omega^{(m)}_{-m}$ and $\hat{\omega}_{-m}$,
\[
\bar{a}_m=\mathsf{ave}( \hat{\omega}^{(m)}_{-m}-\hat{\omega}_{-m}).
\]
As both $\hat{\omega}^{(m)}_{-m}$ and $\hat{\omega}_{-m}$ satisfy the constraint in  \eqref{eq:con-MLE-l-m}, we must have $\|\xi-\omega_{-m}^*\|_{\infty}\leq 5$.
Then, we can apply Lemma \ref{le:second-min-eigenvalue} to the subset of the data and obtain:
\begin{equation*}
\begin{split}
&~~\frac{1}{2}( \hat{\omega}^{(m)}_{-m}-\hat{\omega}_{-m}-\bar{a}_m\mathbbm{1}_{r+t-1})^\top\nabla^2\ell_n^{(-m)}(\xi)
( \hat{\omega}^{(m)}_{-m} - \hat{\omega}_{-m} - \bar{a}_m\mathbbm{1}_{r+t-1})
\\
\geq &~~ \frac{1}{8\kappa}rp\|\hat{\omega}^{(m)}_{-m}-\hat{\omega}_{-m}-\bar{a}_m\mathbbm{1}_{r+t-1}\|_2^2.
\end{split}
\end{equation*}
Using the Cauchy--Schwarz inequality, we obtain:
\begin{align}\label{eq:yinli1}
\|\hat{\omega}^{(m)}_{-m}-\hat{\omega}_{-m}-\bar{a}_m\mathbbm{1}_{r+t-1}\|_2\leq \frac{\|\nabla\ell^{(-m)}(\hat{\omega}_{-m})\|_2}{\frac{1}{8\kappa }rp}.
\end{align}

We now bound $\|\nabla\ell^{(-m)}(\hat{\omega}_{-m})\|_2$.
From the definitions of ${\ell}^{(-m)}(\omega_{-m})$, $\nabla\ell(\omega)$ and $\nabla{\ell}^{(m)}(\omega_m|\omega_{-m})$ in \eqref{eq:MLE-new},
we have:
\begin{align*}
\nabla {\ell}^{(-m)}(\omega_{-m}) = \nabla\ell(\omega) - \nabla{\ell}^{(m)}(\omega_m|\omega_{-m}).
\end{align*}
Because $\nabla\ell(\hat{\omega})=0$, we have:
\begin{align*}
\nabla{\ell}^{(-m)}(\omega_{-m})_{|\omega  = \hat \omega }=-\nabla{\ell}^{(m)}(\omega_m|\omega_{-m})_{|\omega  = \hat \omega },
\end{align*}
Where:
\begin{align*}
\nabla {\ell ^{(m)}}{({\omega _m}|{\omega _{ - m}})_{|\omega  = \hat \omega }} = \sum\limits_{i \in [r]} {\frac{{\partial {\ell ^{(m)}}{{({\omega _m}|{\omega _{ - m}})}_{|\omega  = \hat \omega }}}}{{\partial {\omega _i}}}{e_i}}  + \sum\limits_{j \in [t]} {\frac{{\partial {\ell ^{(m)}}{{({\omega _m}|{\omega _{ - m}})}_{|\omega  = \hat \omega }}}}{{\partial {\omega _{j + r}}}}{e_{j + r}}}.
\end{align*}
Together with \eqref{eq:yinli1}, this yields:
\begin{align}\label{eq:yinli1-b}
\|\hat{\omega}^{(m)}_{-m}-\hat{\omega}_{-m}-\bar{a}_m\mathbbm{1}_{r+t-1}\|_2
\leq \frac{ \|{\ell ^{(m)}}{({\hat{\omega} _m}|{\hat{\omega} _{ - m}})}\|_2 }{\frac{1}{8\kappa}rp}.
\end{align}

The following calculations are based on the event $\mathcal A_0$ defined in \eqref{eq:A0}.
Note that, when $m\in [r]$,
\begin{eqnarray*}
\frac{{\partial {\ell ^{(m)}}{{({\omega _m}|{\omega _{ - m}})}_{|\omega  = \hat \omega }}}}{{\partial {\omega _i}}} & = & 0,                                  \quad i\in[r]\backslash\{m\},
\\
\frac{{\partial {\ell ^{(m)}}{{({\omega _m}|{\omega _{ - m}})}_{|\omega  = \hat \omega }}}}{{\partial {\omega _{j + r}}}}
& = & X_{m,j+r}\left({ {a_{m,j+r}}-\mu ({{\hat \omega }_m} - {{\hat \omega }_{j+r}})}\right), \quad    j\in[t].
\end{eqnarray*}
It follows that, when $m\in [r]$,
\begin{eqnarray*}
\| {\ell ^{(m)}}{({\hat{\omega}_m}|{\hat{\omega} _{ - m}})}\|_2^2
& = &  \sum_{j\in[t]}{X_{m,j+r}}\left( {\mu ({{\hat \omega }_m} - {{\hat \omega }_{j+r}}) - {a_{m,j+r}}} \right)^2
\\
& \leq & 2\sum_{j\in[t]}X_{m,j+r}(\mu({{\hat \omega }_m} - {{\hat \omega }_{j+r}})-\mu(\omega_m^*-\omega_{j+r}^*))^2
\\
&  & +  2\sum_{j\in[t]}{X_{m,j+r}}\left( {\mu (\omega_m^*-\omega_{j+r}^*) - {a_{m,j+r}}} \right)^2
\\
&\leq &  \frac{8}{16}\|\hat{\omega}-\omega^*\|_{\infty}^2\sum_{j\in[t]}X_{m,j+r} +2\sum_{j\in[t]}X_{m,j+r}\left( {\mu (\omega_m^*-\omega_{j+r}^*) - {a_{m,j+r}}} \right)^2
\\
& \leq &  \frac{3}{4}tp\|\hat{\omega}-\omega^*\|_{\infty}^2
+2\sum_{j\in[t]}X_{m,j+r}\left( {\mu (\omega_m^*-\omega_{j+r}^*) - {a_{m,j+r}}} \right)^2.
\end{eqnarray*}
Together with \eqref{eq:yinli1-b}, this yields:
\begin{align}
\nonumber
&~~\mathop {\max }\limits_{m \in [r]} \left\| {\omega _{ - m}^{(m)} - {{\hat \omega }_{ - m}} - {{\bar a}_m}{\mathbbm{1}_{r + t - 1}}} \right\|_2^2
\\
\nonumber
\le & ~~\frac{{\mathop {\max }\limits_{m \in [r]} \sum\limits_{j \in [t]} {{X_{m,j + r}}{{\left( {\mu (\omega _m^* - \omega _{j + r}^*) - {a_{m,j + r}}} \right)}^2}} }}{{\frac{1}{{32b_n^2}}{r^2}{p^2}}} + \frac{{\left\| {\hat \omega  - {\omega ^*}} \right\|_\infty ^2}}{{\frac{{{r^2}p}}{{12b_n^2t}}}}
\\
\label{eq-yinli1-c}
\lesssim &~~ \frac{b_n^2\log r}{ rp } + \frac{ b_n^2{\left\| {\hat \omega  - {\omega ^*}} \right\|_\infty ^2}}{ rp }
\end{align}
where the last inequality is due to \eqref{le:ball12-3}.

By the definition of $\hat{\omega}$, $\ell(\hat{\omega})=\min_{\omega:\mathbbm{1}_{r+t}^\top\omega=0}\ell(\omega)$,
Because $\ell(\omega)=\ell(\omega+c\mathbbm{1}_{r+t})$ for any $c\in\mathbb{R}$, $\ell(\hat{\omega})=\min_{\omega}\ell(\omega)$.
Observe that:
$$\ell^{(m)}(\omega_m^*|\hat{\omega}_{-m}) + \ell^{(-m)}(\hat{\omega}_{-m})\geq\ell(\hat{\omega})=\ell^{(m)}(\hat{\omega}_m|\hat{\omega}_{-m}) + \ell^{(-m)}(\hat{\omega}_{-m}).$$
Using a second-order Taylor expansion, we obtain:
\begin{eqnarray}
\label{ineq-two-lm}
\begin{array}{rcl}
\ell^{(m)}(\omega_m^*|\hat{\omega}_{-m}) &\geq & \ell^{(m)}(\hat{\omega}_m|\hat{\omega}_{-m})
\\
&= & \ell^{(m)}(\omega_m^*|\hat{\omega}_{-m}) + (\hat{\omega}_m-\omega_m^*)\varphi ^{(m)}(\omega_m^*|\hat{\omega}_{-m})
+ \frac{1}{2}(\hat{\omega}_m-\omega_m^*)^2\psi^{(m)}(\xi|\hat{\omega}_{-m}),
\end{array}
\end{eqnarray}
where $\xi$ denotes a convex combination of $\omega_m^*$ and $\hat{\omega}_m$;
From Lemma \ref{lemma-existence}, we have:
\begin{align*}
|\xi-\omega_m^*|\leq |\hat{\omega}_m-\omega_m^*|\leq \|\hat{\omega}-\omega^*\|_{\infty}\leq 5.
\end{align*}
Therefore, for any $i\in[r+t]\backslash\{m\}$, we have:
\begin{align*}
|\xi-\hat{\omega}_i|\leq |\xi-\omega_m^*|+|\omega_m^*-\omega_i^*|+|\hat{\omega}_i-\omega_i^*|\leq 10 + \|\omega^*\|_\infty.
\end{align*}
Together with Lemma \ref{le:second-min-eigenvalue} and the expression for $\psi^{(m)}$ in \eqref{def-psi-m}, this yields:
\[
\frac{1}{2}\psi^{(m)}(\xi|\hat{\omega}_{-m})\gtrsim \frac{1}{b_n}rp.
\]
It follows from \eqref{ineq-two-lm} that:
\begin{align}
 \label{eq:fixec}
|\hat{\omega}_m-\omega_m^*| \lesssim \frac{ |\varphi^{(m)}(\omega_m^*|\hat{\omega}_{-m})| }{ (\frac{1}{b_n}rp) }.
\end{align}
We now find that $\varphi^{(m)}(\omega_m^*|\hat{\omega}_{-m})$.
When $m\in [r]$, we have:
\begin{align}
\nonumber
|{\varphi ^{(m)}}(\omega _m^*|{{\hat \omega }_{ - m}})| &= \Big| {\sum\limits_{j \in [t]} {{X_{m,j + r}}\left( { - {a_{m,j + r}} + \mu \left( {\omega _m^* - {{\hat \omega }_{j + r}}} \right)} \right)} } \Big|
\\
\nonumber
& \le
\Big|
\underbrace{
 {\sum\limits_{j \in [t]} {{X_{m,j + r}}\left( { - {a_{m,j + r}} + \mu \left( {\omega _m^* - \omega _{j + r}^*} \right)} \right)} }
}_{S_{m,1}}
\Big|
\\
\nonumber
&\quad +
\Big|
\underbrace{
{\sum\limits_{j \in [t]} {{X_{m,j + r}}\left( { - \mu \left( {\omega _m^* - \omega _{j + r}^*} \right) +
\mu \left( {\omega _m^* - \hat{\omega}_{j + r}^{(m)} + {a_m}} \right)} \right)} }
}_{S_{m,2}}
\Big|
\\
\label{ineq-sm123}
&\quad + \Big|
\underbrace{ {\sum\limits_{j \in [t]} {{X_{m,j + r}}\left( { - \mu \left( {\omega_m^* - \hat{\omega}_{j + r}^{(m)} + {a_m}} \right) + \mu \left( {\omega _m^* - {{\hat \omega }_{j + r}}} \right)} \right)} }
}_{S_{m,3}}
\Big|.
\end{align}
First, we bound $S_{m,2}$. From Lemma \ref{le:ball-p}, we have:
\begin{eqnarray}
\nonumber
|S_{m,2}| & \le &  p\left| {\sum\limits_{j \in [t]} {\left( { - \mu \left( {\omega _m^* - \omega _{j + r}^*} \right) + \mu \left( {\omega _m^* - \hat{\omega}_{j + r}^{(m)} + {a_m}} \right)} \right)} } \right|
\\
\nonumber
 &  &  + \left| {\sum\limits_{j \in [t]} {\left( {{X_{m,j + r}} - p} \right)\left( { - \mu \left( {\omega _m^* - \omega _{j + r}^*} \right) + \mu \left( {\omega _m^* - \hat{\omega}_{j + r}^{(m)} + {a_m}} \right)} \right)} } \right|
 \\
 \nonumber
& = &  p\left| {\sum\limits_{j \in [t]} {\frac{{{e^{{\xi _j}}}}}{{{{\left( {1 + {e^{{\xi _j}}}} \right)}^2}}}\left( {\left( {\omega _m^* - \hat{\omega}_{j + r}^{(m)} + {a_m}} \right) - \left( {\omega _m^* - \omega _{j + r}^*} \right)} \right)} } \right|
\\
\nonumber
 &  &  + \left| {\sum\limits_{j \in [t]} {\left( {{X_{m,j + r}} - p} \right)\frac{{{e^{{\xi _j}}}}}{{{{\left( {1 + {e^{{\xi _j}}}} \right)}^2}}}\left( {\left( {\omega _m^* - \hat{\omega}_{j + r}^{(m)} + {a_m}} \right) - \left( {\omega _m^* - \omega _{j + r}^*} \right)} \right)} } \right|
 \\
 \nonumber
& \lesssim &  \Big\{ \frac{p\sqrt t }{{{c_n}}}  + \frac{1}{c_n} \sqrt{ \sum_{j\in [t]} (X_{m,j+r}-p)^2 } \Big\}
\cdot
{\left\| {\omega _{ - m}^* - \omega _{ - m}^{(m)} + {a_m}{1_{r + t - 1}}} \right\|_2}
\\
\label{ineq-Sm2}
& \lesssim &  \frac{\sqrt{tp} }{{{c_n}}}
\cdot
{\left\| {\omega _{ - m}^* - \omega _{ - m}^{(m)} + {a_m}{1_{r + t - 1}}} \right\|_2},
\end{eqnarray}
where the last inequality is based on Lemma \ref{le:ball-p}.
Similarly, we also have:
\begin{equation}\label{ineq-Sm3}
|S_{m,3}| \lesssim \frac{\sqrt{tp} }{{{c_n}}}
\cdot
{\left\| {\omega _{ - m}^* - \omega _{ - m}^{(m)} + {a_m}{1_{r + t - 1}}} \right\|_2}.
\end{equation}
Combining \eqref{ineq-sm123}, \eqref{ineq-Sm2}, and \eqref{ineq-Sm3} yields:
\begin{align}
\nonumber
{\left\| {\hat \omega  - {\omega ^*}} \right\|_\infty } & \le \frac{{\mathop {\max }\limits_{m \in [r]} \left| {\sum\limits_{j \in [t]} {{X_{m,j + r}}\left( { - {a_{m,j + r}} + \mu \left( {\omega _m^* - \omega_{j + r}^*} \right)} \right)} } \right|}}{{\frac{1}{{4{b_n}}}rp}}
\\
\nonumber
& + \frac{\sqrt{tp} }{ (c_n/b_n) rp }
\cdot
{\left\| {\omega _{ - m}^* - \omega _{ - m}^{(m)} + {a_m}{1_{r + t - 1}}} \right\|_2}
\\
\label{ineq-final-c}
& \lesssim  b_n \sqrt{ \frac{ \log r}{ rp } } + \frac{\sqrt{tp} }{ (c_n/b_n) rp }
\cdot
{\left\| {\omega _{ - m}^* - \omega _{ - m}^{(m)} + {a_m}{1_{r + t - 1}}} \right\|_2}
\end{align}
where the last inequality is due to \eqref{le:ball12-1}.

Recall that $\bar{a}_m=\mathsf{ave}( \hat{\omega}^{(m)}_{-m}-\hat{\omega}_{-m})$
$a_m=\mathsf{ave}({\hat{\omega}}_{-m}^{(m)}-\omega_m^*)$
 $\mathbbm{1}_{r+t}^\top \hat{\omega}=\mathbbm{1}_{r+t}^\top\omega^*=0$.
Therefore, we have:
\begin{align*}
\left\| {{a_m}{\mathbbm{1}_{r + t - 1}} - {{\bar a}_m}{\mathbbm{1}_{r + t - 1}}} \right\|_2^2 &=
 \left\| {\mathsf{ave}(\hat{\omega} _{ - m}^{(m)} - \omega _m^*){\mathbbm{1}_{r + t - 1}} - \mathsf{ave}(\hat{\omega} _{ - m}^{(m)} - {{\hat \omega }_{ - m}}){\mathbbm{1}_{r + t - 1}}} \right\|_2^2\\
& = (r + t - 1){(\mathsf{ave}({{\hat \omega }_{ - m}} - \omega _{ - m}^*))^2} \le \frac{{\left\| {\hat \omega  - {\omega ^*}} \right\|_\infty ^2}}{{r + t - 1}}.
\end{align*}
This, together with \eqref{eq-yinli1-c} and \eqref{ineq-final-c}, yields:
\[
\max_{m\in [r]} \left\| {\omega _{ - m}^* - \omega _{ - m}^{(m)} + {a_m}{1_{r + t - 1}}} \right\|_2^2
\lesssim \frac{ b_n^2 \log r }{ rp }  + \max_{m\in [r]} \frac{ 1 }{ (c_n/b_n)^2 rp } \left\| {\omega _{ - m}^* - \omega _{ - m}^{(m)} + {a_m}{1_{r + t - 1}}} \right\|_2^2
\]
which yields:
\[
\max_{m\in [r]} \left\| {\omega _{ - m}^* - \omega _{ - m}^{(m)} + {a_m}{1_{r + t - 1}}} \right\|_2^2
\lesssim \frac{ b_n^2 \log r }{ rp }.
\]
Substituting this back into \eqref{ineq-final-c} yields \eqref{eq:theorem-con-b}.
This completes the proof.
\end{proof}

\subsection{Proof of Theorem \ref{th:CL}}
\label{subsec-th2-proof}

Before proving Theorem \ref{th:CL}, we give some necessary lemmas.

\begin{lemma}
\label{le:a1a2}
Let $1_{\{\cdot\}}$  denote the indicator function.
For $i\ne j$,  define
\begin{eqnarray}
\begin{array}{c}
\xi _{ij} = \sum\limits_{k \in [r+t-1]} 1_{ \{ X_{i,k} > 0, X_{j,k} > 0 \} }, \\
{E_{n1}}:=\big\{ \min\limits_{i,j \in [r],i \ne j} \xi_{ij} \ge \frac{1}{2}tp^2 \big\},
\quad
{E_{n2}}:=\big\{ \min\limits_{i-r \in [t], j-r\in[t],i \ne j} \xi_{ij} \ge \frac{1}{2}rp^2 \big\},
\end{array}
\end{eqnarray}
Then, we have
\begin{eqnarray}
\label{eq:A1}
\mathbb{P}( {E_{n1}} ) & \ge  & 1 - \frac{{r(r - 1)}}{2}{e^{ - \frac{1}{8}tp^2}},
\\
\label{eq:A2}
\mathbb{P}({E_{n2}})  & \ge   & 1 - \frac{{t(t - 1)}}{2}{e^{ - \frac{1}{8}rp^2}}.
\end{eqnarray}
\end{lemma}

Recall that  $S=(s_{ij})_{i,j\in [r+t]\backslash \{ 1\} }$, where
\[
s_{ij}=\frac{\delta_{ij}}{v_{ii}}+\frac{1}{v_{11}}.
\]
In the above equation, $\delta_{ij}$ is the Kronecker delta function, i.e., $\delta_{ij}=1$ if $i = j$; otherwise, $\delta_{ij}=0$.

\begin{lemma}
\label{le:V-matrix-bound}
Assume that $1/b_n \le  \mu^\prime( \alpha_i - \beta_j ) \le 1/c_n$ for all $i\in [r]$ and $j\in [t]$, where $b_n \ge c_n \ge 4$.
Under Condition \ref{condi-t-r},  we have
\begin{equation}
\label{eq:V-matrix-bound-a}
\| V^{-1} - S \|_{\max} \le \frac{  12 b_n^3  }{  r^2 p^2 c_n^2  },
\end{equation}
with probability $1 - O( r^{-5} )$.
\end{lemma}

\begin{lemma}
\label{lemma-mul-clt}
Let $\bm{d}=(d_2,\ldots,d_{r+t})^\top$.
For two fixed integer $k_1$ and $k_2$,
the vector $ ( [S(\bm{d}-{\mathbb{E}}(\bm{d}))]_1, \ldots, [S(\bm{d}-{\mathbb{E}}(\bm{d}))]_{k_1},
[S(\bm{d}-{\mathbb{E}}(\bm{d}))]_{r+1}, \ldots, [S(\bm{d}-{\mathbb{E}}(\bm{d}) ) ]_{r+k_2} )^\top$
asymptotically follow a multivariate normal distribution with the mean $\mathbf{0}$ and the covariance matrix $\Sigma=(\sigma_{ij}^2)$ with the diagonal elements
\[
\sigma_{ii}^2  = \begin{cases} \frac{ 1 }{ v_{ii} } + \frac{1}{ v_{11} },  &   i \in [k_1],
\\
 \frac{ 1 }{ v_{i+r, i+r} } + \frac{1}{ v_{11} },  &   i \in [k_2],
\end{cases}
\]
and all non-diagonal elements $1/v_{11}$.
\end{lemma}

We are now ready to prove Theorem \ref{th:CL}.

\begin{proof}[Proof of Theorem \ref{th:CL}]
By the second-order Taylor expansion, for $i\in[r]$ and $j\in[t]$, we have
\begin{eqnarray}
\nonumber
 \mu ({{\hat \theta }_i} - {{\hat \theta }_{j+r}}) - \mu (\theta _i^* - \theta _{j+r}^*)
& = &  \mu '(\theta _i^* - \theta _{j+r}^*)\left({({{\hat \theta }_i} - {{\hat \theta }_{j+r}}) - (\theta _i^* - \theta _{j+r}^*)}\right)
\\
\label{eq:omega_tai}
& & +   \frac{1}{2}\mu ''(\tau_{ij}){\left[ {({{\hat \theta }_i} - {{\hat \theta }_{j+r}}) - (\theta _i^* - \theta _{j+r}^*)} \right]^2},
\end{eqnarray}
where  $\tau_{ij} = \theta _i^* - \theta _{j+r}^* + \gamma_{ij} \left[ {({{\hat \theta }_i} - {{\hat \theta }_{j+r}}) - (\theta _i^* - \theta _{j+r}^*)} \right]$ with $\gamma_{ij}\in [0,1]$.
For simplicity, let
\begin{equation}
\label{eq:gij}
g_{i,j+r}=\frac{1}{2}{X_{i,j+r}}\mu ''({{\tilde \theta }_i} - {{\tilde \theta }_{j+r}}){\left[ {({{\hat \theta }_i} - {{\hat \theta }_{j+r}}) - (\theta _i^* - \theta _{j+r}^*)} \right]^2},
\end{equation}
and
\begin{equation}
\label{definition:gi}
\begin{array}{lcr}
{g_i}  & = &  \sum\limits_{j\in[t]} {{g_{i,j+r}}},  \quad i \in [r],
\\
{g_{j+r}} & = &  \sum\limits_{i\in[r]\backslash\{1\}} {{g_{i,j+r}}}, \quad j \in [t].
\end{array}
\end{equation}
Further, let $\bm{g} = ({g_2}, \ldots ,{g_{r + t}})^{\top}$.
Let
\begin{equation*}
\begin{array}{lcr}
{h_i}-\mathbb{E}^*(h_i) & =  & \sum\limits_{j \in [t]} {{X_{i,j + r}}\mu ({{\hat \theta }_i} - {{\hat \theta }_{j + r}})}-\sum\limits_{j \in [t]} {{X_{i,j + r}}\mu ({{\theta }_i^*} - {{ \theta }_{j + r}^*})}, \quad i\in[r],
\\
h_{j+r}-\mathbb{E}^*(h_{j+r}) & =  & \sum\limits_{i \in [r]} {{X_{i,j + r}}\mu ({{\hat \theta }_i} - {{\hat \theta }_{j + r}})}-\sum\limits_{i \in [r]} {{X_{i,j + r}}\mu ({{\theta }_i^*} - {{ \theta }_{j + r}^*})}, \quad j\in[t],
\end{array}
\end{equation*}
and $\mathbf{h}=(h_2, \ldots, h_{r+t})^\top$, where  $\mathbb{E}^*$ denotes the expectation condition on the sampling matrix $X$.
In what follows, we ignore the superscript ``*" in $\mathbb{E}^*$.
This, together with \eqref{eq-mle-equa} and \eqref{eq:omega_tai}, yields
\begin{equation}
\label{eq:matrix-a}
\bm{h}-{\mathbb{E}}(\bm{h}) = V(\hat \theta  - {\theta ^*}) + \bm{g},
\end{equation}
which is equivalent to
\begin{align}
\hat{\theta}  - \theta^* = {V^{ - 1}}(\bm{h}-{\mathbb{E}}(\bm{h}) ) + {V^{ - 1}}\bm{g} = S(\bm{h}-{\mathbb{E}}(\bm{h}) )
+ (V^{-1}-S)(\bm{h}-{\mathbb{E}}(\bm{h}) )
+ {V^{ - 1}}\bm{g}.
\end{align}
In view of Lemma \ref{lemma-mul-clt}, it is sufficient to prove
\begin{eqnarray}
\label{eq-aim-th2-a}
\big\{ (V^{-1}-S)(\bm{h}-{\mathbb{E}}(\bm{h}) ) \big\}_i & = & o_p( r^{-1/2} ),
\\
\label{eq-aim-th2-b}
 {\left\| {{V^{ - 1}}\bm{g}} \right\|_\infty } & = & {O_p}\left( {\frac{{ \log r}}{{r{p^2}}}} \right) = o_p( r^{-1/2} ).
\end{eqnarray}

We first show \eqref{eq-aim-th2-a}. Let $Z=V^{-1}-S$ and $\bar{h}_i = h_i - \mathbb{E} h_i$.
\begin{eqnarray}
\mathbb{P}( | \sum_{j=1}^{r+t} Z_{ij} \bar{h}_j | \ge \epsilon | X )
& \le & \mathbb{P}( | \sum_{j=1}^{r} Z_{ij} \bar{h}_j | \ge \frac{\epsilon}{2} | X )
+ \mathbb{P}( | \sum_{j=r+1}^{r+t} Z_{ij} \bar{h}_j | \ge \frac{\epsilon}{2} | X ).
\end{eqnarray}
Since $\sum_{j=1}^{r} Z_{ij} \bar{h}_j = \sum_{j=1}^r \sum_{k=r+1}^{r+t} Z_{ij} X_{j k}(a_{j,k}-\mathbb{E} a_{j,k})$,
$\sum_{j=1}^{r} Z_{ij} \bar{h}_j$ can be viewed as the sum of $\sum_{j=1}^r \sum_{k=r+1}^{r+t}X_{j k} $ bounded independent random variables
conditional $X$.
By \citeauthor{hoeffding1963probability}'s (\citeyear{hoeffding1963probability}) inequality, we have
\[
| \sum_{j=1}^r \sum_{k=r+1}^{r+t} Z_{ij} X_{j k}(a_{j,k}-\mathbb{E} a_{j,k}) |
\le 2\sqrt{ \|Z\|_{\max}^2 \sum_{i=1}^r\sum_{j=1}^t X_{i,r+j} \log r}.
\]
with probability at least $1-O(r^{-2})$.  By \citeauthor{chernoff1952measure}'s (\citeyear{chernoff1952measure}) bound, we have
\[
\sum_{i=1}^r\sum_{j=1}^t X_{i,r+j} \le \frac{3}{2} rtp,
\]
with probability at least $ 1- \exp( -rtp/8)$. Combining these,
\[
| \sum_{j=1}^{r} Z_{ij} \bar{d}_j | \le C \|Z\|_{\max}\sqrt{ rtp\log r},
\]
with probability at least $1-O(r^{-2})$. By Lemma \ref{le:V-matrix-bound}, we have
\[
\sum_{j=1}^{r} Z_{ij} \bar{d}_j = O_p\left( \frac{ \sqrt{p\log r} }{ rp^2 } \right) = o_p( r^{-1/2}),
\]
if $p \ge r^{1/4}/\log r$. This shows \eqref{eq-aim-th2-a}.

We now show \eqref{eq-aim-th2-b}.
Note that $\mu^{\prime\prime}(x) = (1-e^x)e^x/( 1 + e^x )^3$.
By Theorem \ref{th:MLE-main}, we have
\begin{equation}
\label{eq:omega_infty}
\| {\hat \theta  - {\theta ^*}} \|_\infty
 = {O_p}\left( {\sqrt {\frac{{\log r}}{{rp}}} } \right).
\end{equation}
Therefore, for $i\in [r], j\in [t]$, we have
\begin{equation*}
 \mu^{\prime\prime}( \tau_{ij} )  \le  \frac{ e^{ \tau_{ij} } }{  ( 1 + e^{ \tau_{ij} } )^2 } \le \frac{1}{4}.
\end{equation*}
According to the definition of $g_{i,j+r}$ in \eqref{eq:gij}, it follows that
\begin{align}
| {{g_{i,j+r}}} | \le \frac{1}{{{2}}}\| {\hat \theta  - {\theta ^*}} \|_\infty ^2.
\end{align}
Therefore,
\begin{equation}
\label{eq:max_gi}
\mathop {\max }\limits_{i\in[r+t]\backslash\{1\}} \left| {{g_i}} \right| = \frac{3}{{{4}}}tp \times {O_p}\left( {\frac{{\log r}}{{rp}}} \right) =  {O}\left( {\frac{{t\log r}}{{r{c_n}}}} \right).
\end{equation}
By a second order Taylor expansion, we have
\begin{equation*}
\begin{split}
{h_1}-{\mathbb{E}}({h_1}) =  - \sum\limits_{j \in [ t] } \Big\{ \mu^{\prime}(\theta_i^*-\theta_{j+r}^*)({{\hat \theta }_{j+r}} - \theta_{j+r}^*)
 + \frac{1}{2}\mu^{\prime\prime}( \tau_{i,j+r} ) ( \hat{\theta}_i - \theta_i^*)^2 \Big\},
\end{split}
\end{equation*}
where $\tau_{i,j+r}$ is a convex combination of $\hat{\theta}_i-\theta_i^*$ and $\hat{\theta}_{j+r}-\theta^*_{j+r}$.
The following calculations are based on the event $\mathcal A_0$ define in \eqref{eq:A0}.
By \eqref{eq:omega_infty}, we have
\begin{align*}
| (h_1- \mathbb{E}(h_1))  + \sum\limits_{j \in [ t] } \Big\{ \mu^{\prime}(\theta_i^*-\theta_{j+r}^*)({{\hat \theta }_{j+r}} - \theta_{j+r}^*) |
= O\left( \frac{ d_{\max} \log r}{ rp} \right) = O\left( \frac{ t\log r}{r} \right).
\end{align*}
According to the definition of $g_i$, we have
\begin{align*}
\sum\limits_{i \in [r + t]\backslash\{1\}} {{g_i}}  = {h_1} - {\mathbb{E}}(h_1) + \sum\limits_{i \in [r + t]\backslash\{1\}} {{v_{i1}}({{\hat \theta }_i} - \theta _i^*)} .
\end{align*}
It follows that
\begin{align}\label{eq:sum_gi}
\sum\limits_{i \in [r + t]\backslash\{1\}} {{g_i}}  = {O}\left( {\frac{{t\log r}}{{r}}} \right).
\end{align}
Combining \eqref{eq:V-matrix-bound-a}, \eqref{eq:max_gi} and  \eqref{eq:sum_gi}, it yields
\begin{align*}
&~ \| V^{-1} \bm{g} \|_\infty  \le \| S\bm{g} \|_\infty + \| ( V^{-1} - S ) \bm{g} \|_\infty
\\
\le &~  \max_{ i\in[r+t]\backslash\{1\} } \frac{| g_i |}{ v_{i,i} }  +
  \frac{1}{v_{1,1}} | \sum\limits_{ i \in [r + t]\backslash\{1\}} g_i  |
  + ( r+ t)\| V^{-1} - S \|_{\max} \| \bm{g} \|_\infty
  \\
  \lesssim & ~
  \frac{ b_n }{rp } \cdot \frac{ t\log r}{ c_n r} + (r+t)\cdot \frac{ b_n^3 }{ c_n^2 r^2 p^2 } \cdot \frac{ t\log r }{ c_n r }
\\
\lesssim & ~ \frac{ b_n^3\log r }{ c_n^3 rp^2},
\end{align*}
where $b_n$ and $c_n$ are defined in Lemma \ref{le:V-matrix-bound}. If $\kappa=O(1)$, then $b_n=O(1)$ and $c_n=O(1)$.
This shows \eqref{eq-aim-th2-b} if $p\gg (r/\log r)^{1/4}$.
It completes the proof.
\end{proof}

\setlength{\itemsep}{-1.5pt}
\setlength{\bibsep}{0ex}
\bibliographystyle{apalike}
\bibliography{ref}

\begin{thebibliography}{}

\bibitem[Addison et~al., 2020]{riiid-test-answer-prediction}
Addison, H., Lee, C., Shin, D., Jeon, H.~P., Baek, J., Chang, K., NHeffernan,
  K., Dane, S.~S., and Lee, Y. (2020).
\newblock Riiid answer correctness prediction.

\bibitem[Agresti, 1993]{agresti1993computing}
Agresti, A. (1993).
\newblock Computing conditional maximum likelihood estimates for generalized
  rasch models using simple loglinear models with diagonals parameters.
\newblock {\em Scandinavian Journal of Statistics}, pages 63--71.

\bibitem[Andersen, 1972]{andersen1972numerical}
Andersen, E.~B. (1972).
\newblock The numerical solution of a set of conditional estimation equations.
\newblock {\em Journal of the Royal Statistical Society Series B: Statistical
  Methodology}, 34(1):42--54.

\bibitem[Andersen, 1973]{andersen1973conditional}
Andersen, E.~B. (1973).
\newblock {\em Conditional inference and models for measuring}.
\newblock Mentalhygienisk Forlag, Copenhagen.

\bibitem[Birdsall, 2011]{birdsall2011implementing}
Birdsall, M. (2011).
\newblock Implementing computer adaptive testing to improve achievement
  opportunities.
\newblock {\em Office of Qualifications and Examinations Regulation Report}.

\bibitem[Bond and Fox, 2013]{bond2013applying}
Bond, T.~G. and Fox, C.~M. (2013).
\newblock {\em Applying the Rasch model: Fundamental measurement in the human
  sciences.}
\newblock Psychology Press.

\bibitem[Bradley and Terry, 1952]{bradley-terry1952}
Bradley, R.~A. and Terry, M.~E. (1952).
\newblock Rank analysis of incomplete block designs the method of paired
  comparisons.
\newblock {\em Biometrika}, 39(3-4):324--345.

\bibitem[Chatterjee et~al., 2011]{Chatterjee2011random}
Chatterjee, S., Diaconis, P., and Sly, A. (2011).
\newblock {Random graphs with a given degree sequence}.
\newblock {\em Ann. Appl. Probab.}, 21(4):1400--1435.

\bibitem[Chen et~al., 2022]{chen2022partial}
Chen, P., Gao, C., and Zhang, A.~Y. (2022).
\newblock Partial recovery for top-k ranking: optimality of mle and
  suboptimality of the spectral method.
\newblock {\em The Annals of Statistics}, 50(3):1618--1652.

\bibitem[Chen et~al., 2019]{chen2019spectral}
Chen, Y., Fan, J., Ma, C., et~al. (2019).
\newblock Spectral method and regularized mle are both optimal for top-k
  ranking.
\newblock {\em The Annals of statistics}, 47(4):2204.

\bibitem[Chen et~al., 2023a]{JMLR-Chen-2023}
Chen, Y., Li, C., Ouyang, J., and Xu, G. (2023a).
\newblock Statistical inference for noisy incomplete binary matrix.
\newblock {\em Journal of Machine Learning Research}, 24(95):1--66.

\bibitem[Chen et~al., 2023b]{chen2023SS}
Chen, Y., Li, X., Liu, J., and Ying, Z. (2023+b).
\newblock Item response theory-a statistical framework for educational and
  psychological measurement.
\newblock {\em Statistical science}, To appear.

\bibitem[Chernoff, 1952]{chernoff1952measure}
Chernoff, H. (1952).
\newblock A measure of asymptotic efficiency for tests of a hypothesis based on
  the sum of observations.
\newblock {\em The Annals of Mathematical Statistics}, 23(4):493--507.

\bibitem[De~Leeuw and Verhelst, 1986]{de1986maximum}
De~Leeuw, J. and Verhelst, N. (1986).
\newblock Maximum likelihood estimation in generalized rasch models.
\newblock {\em Journal of educational statistics}, 11(3):183--196.

\bibitem[Erd{\H{o}}s and R{\'e}nyi, 1960]{erdHos1960evolution}
Erd{\H{o}}s, P. and R{\'e}nyi, A. (1960).
\newblock On the evolution of random graphs.
\newblock {\em Publ. math. inst. hung. acad. sci}, 5(1):17--60.

\bibitem[Fan et~al., 2023]{fan2023asymptotic}
Fan, Y., Jiang, B., Yan, T., and Zhang, Y. (2023).
\newblock Asymptotic theory in bipartite graph models with a growing number of
  parameters.
\newblock {\em Canadian Journal of Statistics}, 51(4):919--942.

\bibitem[Fischer, 1974]{fischer1974einfuhrung}
Fischer, G. (1974).
\newblock {\em Einf{\"u}hrung in die Theorie psychologischer Tests.}
\newblock Bern: Huber.

\bibitem[Fischer, 1978]{fischer1978probabilistic}
Fischer, G. (1978).
\newblock Probabilistic test models and their applications.
\newblock {\em German Journal of Psychology}, 2:298--319.

\bibitem[Fischer, 1995]{fischer1995derivations}
Fischer, G.~H. (1995).
\newblock Derivations of the rasch model.
\newblock In {\em Rasch models: Foundations, recent developments, and
  applications}, pages 15--38. Springer.

\bibitem[Follmann, 1988]{follmann1988consistent}
Follmann, D. (1988).
\newblock Consistent estimation in the rasch model based on nonparametric
  margins.
\newblock {\em Psychometrika}, 53(4):553--562.

\bibitem[Ghosh, 1995]{ghosh1995inconsistent}
Ghosh, M. (1995).
\newblock Inconsistent maximum likelihood estimators for the rasch model.
\newblock {\em Statistics \& Probability Letters}, 23(2):165--170.

\bibitem[G{\"u}rer and Draxler, 2023]{gurer2023penalization}
G{\"u}rer, C. and Draxler, C. (2023).
\newblock Penalization approaches in the conditional maximum likelihood and
  rasch modelling context.
\newblock {\em British Journal of Mathematical and Statistical Psychology},
  76(1):154--191.

\bibitem[Haberman, 1977]{haberman1977maximum}
Haberman, S.~J. (1977).
\newblock Maximum likelihood estimates in exponential response models.
\newblock {\em The Annals of statistics}, 5(5):815--841.

\bibitem[Han et~al., 2020]{han2020asymptotic}
Han, R., Ye, R., Tan, C., and Chen, K. (2020).
\newblock {Asymptotic theory of sparse BradleyšCTerry model}.
\newblock {\em The Annals of Applied Probability}, 30(5):2491 -- 2515.

\bibitem[Hoeffding, 1963]{hoeffding1963probability}
Hoeffding, W. (1963).
\newblock Probability inequalities for sums of bounded random variables.
\newblock {\em Journal of the American Statistical Association},
  58(301):13--30.

\bibitem[Kantorovich and Akilov, 1964]{akilov1964functional}
Kantorovich, L.~V. and Akilov, G.~P. (1964).
\newblock {\em Functional Analysis in Normed Spaces (translated by D.G.
  Brown)}.
\newblock Pergamon press, Oxford.

\bibitem[Lounici, 2008]{Lounici2008}
Lounici, K. (2008).
\newblock {Sup-norm convergence rate and sign concentration property of Lasso
  and Dantzig estimators}.
\newblock {\em Electronic Journal of Statistics}, 2(none):90 -- 102.

\bibitem[Rasch, 1960]{georg1960probabilistic}
Rasch, G. (1960).
\newblock {\em Probabilistic models for some intelligence and attainment
  tests}.
\newblock Copenhagen: Institute of Education Research.

\bibitem[Robitzsch, 2021]{robitzsch2021comprehensive}
Robitzsch, A. (2021).
\newblock A comprehensive simulation study of estimation methods for the rasch
  model.
\newblock {\em Stats}, 4(4):814--836.

\bibitem[Rutkowski et~al., 2013]{rutkowski2013handbook}
Rutkowski, L., von Davier, M., and Rutkowski, D. (2013).
\newblock {\em Handbook of international large-scale assessment: Background,
  technical issues, and methods of data analysis}.
\newblock CRC Press.

\bibitem[Wang et~al., 2022]{wang2022two}
Wang, Q., Yan, T., Jiang, B., and Leng, C. (2022).
\newblock Two-mode networks: inference with as many parameters as actors and
  differential privacy.
\newblock {\em Journal of Machine Learning Research}, 23(292):1--38.

\bibitem[Wright and Mok, 2004]{wright2004overview}
Wright, B.~D. and Mok, M.~C. (2004).
\newblock An overview of the family of rasch measurement models.
\newblock In {\em In Smith E., Smith R. (ed.): Introduction to Rasch
  Measurement}. Jampress, Chicago.

\bibitem[Wu et~al., 2020]{wu2020variational}
Wu, M., Davis, R.~L., Domingue, B.~W., Piech, C., and Goodman, N. (2020).
\newblock Variational item response theory: Fast, accurate, and expressive.
\newblock {\em arXiv preprint arXiv:2002.00276}.

\bibitem[Yan et~al., 2016]{yan2016asymptotics}
Yan, T., Leng, C., and Zhu, J. (2016).
\newblock Asymptotics in directed exponential random graph models with an
  increasing bi-degree sequence.
\newblock {\em The Annals of Statistics}.

\bibitem[Yan and Xu, 2013]{Yan2013clt}
Yan, T. and Xu, J. (2013).
\newblock {A central limit theorem in the $\beta$-model for undirected random
  graphs with a diverging number of vertices}.
\newblock {\em Biometrika}, 100(2):519--524.

\bibitem[Yang and Ma, 2024]{yang2024randompairing}
Yang, Y. and Ma, C. (2024).
\newblock Random pairing mle for estimation of item parameters in rasch model.

\end{thebibliography}

\end{document}